\documentclass[pdflatex,sn-mathphys-num]{sn-jnl}

\usepackage{graphicx}%
\usepackage{multirow}%
\usepackage{amsmath,amssymb,amsfonts}%
\usepackage{amsthm}%
\usepackage{mathrsfs}%
\usepackage[title]{appendix}%
\usepackage{xcolor}%
\usepackage{textcomp}%
\usepackage{manyfoot}%
\usepackage{booktabs}%
\usepackage{algorithm}%
\usepackage{algorithmicx}%
\usepackage{algpseudocode}%
\usepackage{listings}%
\usepackage{enumerate}%
\usepackage{bm}%
\usepackage{diagbox}%

\allowdisplaybreaks[3]

\numberwithin{equation}{section}

\theoremstyle{thmstyleone}%
\newtheorem{theorem}{Theorem}[section]
\newtheorem{tm}[theorem]{Theorem}
\newtheorem{pr}[theorem]{Proposition}%
\newtheorem{lm}[theorem]{Lemma}%
\newtheorem{co}[theorem]{Corollary}%

\theoremstyle{thmstyletwo}%
\newtheorem{re}[theorem]{Remark}%

\theoremstyle{thmstylethree}%
\newtheorem{df}[theorem]{Definition}%

\newcommand{\R}{\mathbb{R}}
\newcommand{\N}{\mathbb{N}}
\newcommand{\Z}{\mathbb{Z}}
\newcommand{\dd}{\mathrm{d}}

\newcommand{\ve}{\varepsilon}
\newcommand{\dis}{\displaystyle}

\newcommand{\ol}{\overline}

\newcommand{\Var}{\mathrm{Var}}

\newcommand{\nn}{\nonumber}

\newcommand{\E}{\mathbb{E}}

\newcommand{\bx}{\bm{x}}
\newcommand{\I}{\mathrm{I}}
\newcommand{\II}{\mathrm{I\hspace{-.1em}I}}
\newcommand{\III}{\mathrm{I\hspace{-.1em}I\hspace{-.1em}I}}
\newcommand{\IV}{\mathrm{I\hspace{-.1em}V}}

\begin{document}

\title[Fractional binomial distributions and their applications]{Fractional binomial distributions induced by
the generalized binomial theorem and their applications}


\author*[1]{\fnm{Masanori} \sur{Hino}}\email{hino@math.kyoto-u.ac.jp}

\author[2]{\fnm{Ryuya} \sur{Namba}}\email{rnamba@cc.kyoto-su.ac.jp}

\affil[1]{\orgdiv{Department of Mathematics, Graduate School of Science}, \orgname{Kyoto University}, \orgaddress{\city{Kyoto}, \postcode{606-8502}, \country{Japan}}}

\affil[2]{\orgdiv{Department of Mathematics, Faculty of Science}, \orgname{Kyoto Sangyo University}, \orgaddress{\city{Kyoto}, \postcode{603-8555}, \country{Japan}}}

\abstract{We develop a fractional extension of the classical binomial distribution and the associated Bernstein operator, formulated within the framework of the generalized binomial theorem (Hara and Hino [Bull.\ London Math.\ Soc.\ {\bf 42} (2010), 467--477]).
This provides a new probabilistic structure not representable as the law of 
the sum of independent and identically distributed random variables. 
Despite this nonstandard nature, 
we establish several of its fundamental analytic and probabilistic properties, 
including limit theorems,
through a unified framework based on the generalized binomial theorem.
We further analyze the properties of the fractional Bernstein operator associated with the fractional binomial distribution. In particular, we prove that the iterates of the operator converge to a generalized Wright--Fisher diffusion semigroup after a proper diffusive rescaling.
}




\maketitle

\tableofcontents

\section{{\bf Introduction}}
\label{Sect:Introduction}

\subsection{Fractional binomial distribution and fractional Bernstein operator}
The binomial distribution and the associated Bernstein operator form one of the classical bridges between probability and approximation theory.
The classical binomial distribution  
    \[
        \mu_x^{(n)}(\dd z):=\sum_{j=0}^n \binom{n}{j}x^j(1-x)^{n-j}\delta_{j}(\dd z), \qquad n \in \N,\ x \in (0, 1),
    \]
arises naturally from the classical binomial theorem
$(x+y)^n=\sum_{j=0}^n \binom{n}{j}x^jy^{n-j}$. 
Here, $\delta_j$ denotes the Dirac measure at $j$. 
It corresponds to the law of the sum of independent and 
identically distributed (i.i.d.~in short) Bernoulli random variables, 
and induces the celebrated  
\emph{Bernstein operator}, acting on 
the set $C([0, 1])$ consisting of continuous functions on $[0, 1]$, 
defined as
    \begin{equation*}
        B_nf(x):=\int_{\R} f\left(\frac{z}{n}\right) \, \mu_{x}^{(n)}(\dd z)
        =\sum_{j=0}^n \binom{n}{j}x^j(1-x)^{n-j}f\left(\frac{j}{n}\right), \qquad f \in C([0, 1]).
    \end{equation*}
Bernstein proved in \cite{Bernstein}
that  
$\|B_n f - f\|_\infty \to 0$ holds as $n \to \infty$
for $f \in C([0, 1])$, where $\|\cdot\|_\infty$ denotes the supremum norm. 
We refer to e.g., \cite{Bustamante, Altomare, Alt-C} for further details. 
Moreover, 
    \[
    B_n^k:=\underbrace{B_n \circ \cdots \circ B_n}_{k}, \qquad k \in \N,
    \] 
forms a discrete semigroup, whose certain scaling limit as $n \to \infty$ yields 
the Wright--Fisher diffusion semigroup studied in population genetics. See \cite{KZ, EK, KYZ18}. 

While such a classical theory has been studied intensively and extensively in probability theory, 
the \emph{fractional} analogues of such structures remain largely unexplored. 
It is worth noting that the \emph{generalized binomial theorem} obtained by Hara--Hino~\cite{HH} provides a natural framework 
for such an extension. 
In general, for $z \in \mathbb{C} \setminus \{x \in \R \mid x \le 0\}$ 
and $\gamma \in \R$,
the power $z^\gamma$ is defined as $\exp(\gamma \,\mathrm{Log}\,z)$, 
where $\mathrm{Log}\,z$ is the principal value of $\log z$ 
so that $\mathrm{Log}\,1=0$. 
For $\alpha>0$, put 
\begin{equation}\label{Eq:definition-K_alpha}
    K_\alpha:=\{e^{i\theta} \in \mathbb{C} 
    \mid -\pi < \theta \le \pi, \, e^{i \theta \alpha}=1\}. 
\end{equation}
Note that $K_\alpha=\{1\}$ when $\alpha \in (0, 2)$ and 
that $-1 \in K_\alpha$ if and only if $\alpha/2 \in \N$. 
We define
    \[
    \binom{w}{z}:=\frac{\Gamma(w+1)}{\Gamma(z+1)\Gamma(w-z+1)}
    \]
for $w \in \mathbb{C} \setminus \{-k \mid k \in \N\}$
and $z \in \mathbb{C}$ in general, 
where $\Gamma(\cdot)$ denotes the gamma function. 
If $\infty$ appears in the denominator of the right-hand side of the above, 
then $\binom{w}{z}$ is regarded as zero. 
Also, $0^0$ is defined as 1. 
The generalized binomial theorem is stated as follows.

\begin{pr}[{\cite[Theorem~3.2]{HH}}]
\label{Prop:generalization-binomial-theorem-HH}
    For $\alpha>0$, $n \in \N$, and $\lambda \in (0, 1]$, 
    we have 
    \begin{align}\label{Eq:Generalized-binomial-theorem-HH}
        \alpha\sum_{j=0}^n \binom{\alpha n}{\alpha j} \lambda^{\alpha j}
        &= \sum_{\omega \in K_\alpha}(1+\lambda \omega)^{\alpha n}
            -\frac{\alpha \lambda^\alpha \sin \alpha \pi}{\pi}
            \int_0^1 t^{\alpha-1}(1-t)^{\alpha n} \nonumber \\
        &\quad\times \left\{
        \frac{1}{|t^\alpha-\lambda^\alpha e^{-i\alpha \pi}|^2}
        +\frac{\lambda^{\alpha n}}{|e^{-i\alpha \pi}-(\lambda t)^\alpha|^2}
        \right\} \, \dd t. 
    \end{align}
    Here, the second term on the right-hand side is regarded as zero if $\alpha/2 \in \N$. 
\end{pr}

We note that Proposition~\ref{Prop:generalization-binomial-theorem-HH} also holds for $\lambda>0$ by analytic continuation in the parameter $\lambda$. 
For proving Proposition~\ref{Prop:generalization-binomial-theorem-HH}, 
the authors established the fractional-order Taylor-like series expansions 
with residual terms in \cite[Theorem~3.1]{HH}, which has the fractional calculus as its theoretical background.  
See Proposition~\ref{Prop:Fractional-Taylor-series-HH} for further details.
Incidentally, Proposition~\ref{Prop:generalization-binomial-theorem-HH}
is used to answer Lyons' conjecture relating to 
the \emph{neo-classical inequality}, which plays a crucial role 
in the proof of a fundamental theorem of rough path theory (see \cite[Theorem~2.2.1]{Lyons}).

Let $\alpha>0$ and $n \in \N$. 
By putting $\lambda=x/(1-x)$, $x \in (0, 1)$, in \eqref{Eq:Generalized-binomial-theorem-HH} and multiplying both sides by $(1-x)^{\alpha n}$, we have 
    \begin{align}
    Z_{\alpha, x}^{(n)}:\!
    &=\alpha\sum_{j=0}^n \binom{\alpha n}{\alpha j} x^{\alpha j}(1-x)^{\alpha(n-j)} \nn \\
    &= 1+\sum_{\omega \in K_\alpha \setminus \{1\}}(1-x+x\omega)^{\alpha n}
    -\frac{\alpha x^\alpha (1-x)^\alpha \sin \alpha \pi}{\pi}
    \int_0^1 t^{\alpha-1}(1-t)^{\alpha n} \nonumber \nn \\
    &\quad\times \left\{
    \frac{(1-x)^{\alpha n}}{|(t(1-x))^\alpha-x^\alpha e^{-i\alpha \pi}|^2}
    +\frac{x^{\alpha n}}{|(1-x)^\alpha e^{-i\alpha \pi}-(tx)^\alpha|^2}
    \right\} \, \dd t.
    \label{Eq:definition-normalized-constant}
\end{align}
Keeping in mind that the right-hand side of \eqref{Eq:definition-normalized-constant} 
is close to 1 when $n$ is sufficiently large,
we undertake a systematic study of the fractional binomial distribution and its associated fractional Bernstein operator,
which are defined as follows. 

\begin{df}[Fractional binomial distribution]
\label{Def:fractional-binomial-distribution}
    Let $\alpha > 0$, $n \in \N$, and $x \in [0, 1]$. 
    Then, a probability measure $\mu_{\alpha, x}^{(n)}$ defined as
    \[
        \mu_{\alpha, x}^{(n)}(\dd z)
        :=\frac{\alpha}{Z_{\alpha, x}^{(n)}}\sum_{j=0}^n
        \binom{\alpha n}{\alpha j} x^{\alpha j}(1-x)^{\alpha(n-j)}\delta_j(\dd z)
    \]
    is called the $\alpha$-fractional binomial distribution. 
\end{df}

    Note that
    $Z_{\alpha, 0}^{(n)}=Z_{\alpha, 1}^{(n)}=\alpha$, 
    $\mu_{\alpha, 0}^{(n)}=\delta_0$, and $\mu_{\alpha, 1}^{(n)}=\delta_1$. 

\begin{df}[Fractional Bernstein operator]
\label{Def:fractional-Bernstein-operator}
    Let $\alpha > 0$ and $n \in \N$.
    The $\alpha$-fractional Bernstein operator $B_{\alpha, n}$, acting on $C([0, 1])$,
    is defined by 
        \begin{align*}
        B_{\alpha, n}f(x):\!&=\int_{\R}f\left(\frac{z}{n}\right) \,  
        \mu_{\alpha, x}^{(n)}(\dd z) \\
        &= \frac{\alpha}{Z_{\alpha, x}^{(n)}}\sum_{j=0}^n
        \binom{\alpha n}{\alpha j}
        x^{\alpha j}(1-x)^{\alpha(n-j)}f\left(\frac{j}{n}\right),
        \qquad f \in C([0, 1]), \, x \in [0, 1]. 
        \end{align*}
\end{df}

\noindent
It is clear that $\mu_{1, x}^{(n)}$ and $B_{1, n}$ coincide with 
$\mu_x^{(n)}$ and $B_{n}$, respectively. 
In this sense, the fractional binomial distribution (resp.\ fractional Bernstein operator) is a natural generalization of the binomial distribution (resp.\ Bernstein operator) with additional parameter $\alpha$.
This family interpolates smoothly between the classical and the newly introduced fractional structures; the probability mass function of $\mu_{\alpha, x}^{(n)}$
depends smoothly on both $\alpha$ and $x$.   

The purpose of this paper is to develop a fractional framework that encompasses both $\mu_{\alpha,x}^{(n)}$ and $B_{\alpha,n}$ in a unified manner based on the generalized binomial theorem.
Although many of the resulting statements may appear elementary at first glance, their proofs often require techniques quite different from those used in the classical setting.
These differences arise because, in general, $\mu_{\alpha, x}^{(n)}$ is not expressed as the law of the sum of i.i.d.~random variables.
Our main contribution is to clarify, in a systematic and unified way, the probabilistic and analytic features of $\mu_{\alpha,x}^{(n)}$ and $B_{\alpha,n}$. To the best of our knowledge, no previous work has constructed fractional analogues of the binomial distribution and the Bernstein operator in the unified way developed here.

\subsection{Outline of the paper}
The remainder of the paper is organized as follows.

\smallskip
\noindent
\textbf{The first part (Section~\ref{Sect:fractional-binomial-distribution})}.
        We establish explicit expressions 
        and asymptotics for the moments of $\mu_{\alpha, x}^{(n)}$.  
        In particular, we have 
        \begin{equation}\label{Eq:mean_and_variance}
            \E[S_{\alpha, x}^{(n)}]=nx+O(e^{-\delta n})
            \quad \text{and} \quad 
            \Var(S_{\alpha, x}^{(n)})
            =\frac{1}{\alpha}nx(1-x)+O(e^{-\delta n})
        \end{equation}
        as $n \to \infty$ for some $\delta>0$, where 
        $S_{\alpha, x}^{(n)}$ is a random variable 
        obeying $\mu_{\alpha, x}^{(n)}$.
        This reveals that the variance scales inversely with $\alpha$. 
        See Figure~\ref{Fig:fractional_binomial_distribution}.
        We also obtain a uniform bound of $Z_{\alpha, x}^{(n)}$
        with respect to both $n \in \N$ and $x \in [0, 1]$, 
        and an explicit expression for the characteristic function 
        of $\mu_{\alpha, x}^{(n)}$ with its quantitative estimate.
        
\begin{figure}[t]
  \centering
  \includegraphics[width=12cm]{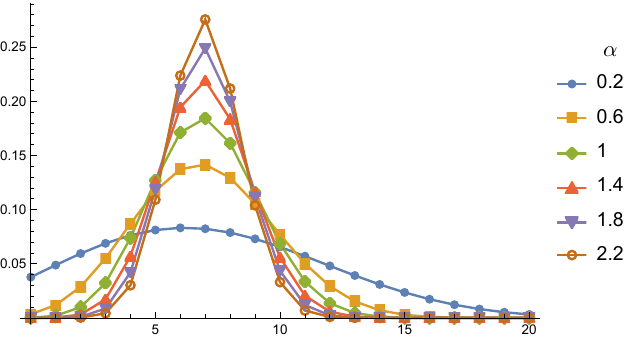}
  \caption{Variation of the fractional binomial distribution $\mu_{\alpha, x}^{(n)}$ with respect to $\alpha$. Shown are the probability mass functions for $n=20$ and $x=0.35$, where the parameter $\alpha$ takes the values $0.2$, $0.6$, $1$, $1.4$, $1.8$, and $2.2$.}
  \label{Fig:fractional_binomial_distribution}
\end{figure}

\smallskip
\noindent
\textbf{The second part (Section \ref{Sect:limit theorems})}.
Based on the properties given in Section~\ref{Sect:fractional-binomial-distribution}, 
we prove several limit theorems: the weak law of large numbers, 
the central limit theorem, and 
the law of small numbers for $\mu_{\alpha, x}^{(n)}$.
In particular, the limiting distribution in the latter case 
is a fractional variant of the Poisson distribution defined in terms of 
the \emph{Mittag-Leffler function} (see e.g., \cite{Herrmann, CO}). 

\smallskip
\noindent
\textbf{The third part (Section~\ref{Sect:fractional Bernstein operator})}.
        We investigate several properties of the fractional Bernstein operator $B_{\alpha, n}$ and its iterates. 
        We prove the uniform approximation result and 
        the Kelisky--Rivlin-type limit theorem for $B_{\alpha, n}$.
        We also prove 
        the convergence of the appropriately scaled operators
        $(B_{\alpha, n}^{\lfloor nt \rfloor})_{n\in\N_0}$, $t \ge 0$, 
        to the Wright--Fisher diffusion semigroup with generator  $\mathcal{L}_\alpha f(x)=(2\alpha)^{-1}x(1-x)f''(x)$  
        on $C^2([0, 1])$ as $n \to \infty$ in some sense. 
        Furthermore, we extend this convergence to the case where
        the fractional order is a continuous function $\alpha(x) \colon [0, 1] \to (0, \infty)$, and the limiting generator 
        is given by 
        $\mathcal{L}_{\alpha(\cdot)}f(x)=(2\alpha(x))^{-1}x(1-x)f''(x)$
        on $C^2([0, 1])$, which leads to  
        more-general diffusion processes on $[0, 1]$.

\subsection{Future perspectives}

We systematically generalize known results for the classical binomial distribution and Bernstein operator into their fractional counterparts by means of a unified framework based on the generalized binomial theorem.
This approach enables us to make substantial contributions to the study of probability distributions and related fields. 
We expect that this framework will shed light on the analytical and geometric aspects of such structures, and deepen the understanding of various mathematical models, including stochastic processes, applied statistics, and information geometry.
For example, 
the fractional binomial distributions may provide 
flexible models for count data which can describe both overdispersion and underdispersion, in view of \eqref{Eq:mean_and_variance}.
Furthermore, the study of the geometric properties of the statistical manifold formed by these distributions may lead to new developments in statistical theory and its applications.

\subsection{Remarks}
    The terminology \emph{fractional} 
    in probability theory is used in several different contexts. 
    For instance, a typical use can be seen in the notion of \emph{fractional Brownian motion}. 
    See \cite{Lee} for the fractional binomial distribution in this sense.
    Yet another use is derived from \emph{fractional calculus}. 
    We refer to \cite{CP} for the fractional binomial process 
    defined in terms of Caputo's fractional derivatives. 
    We also mention some existing notions of linear operators named 
    the {\it fractional} Bernstein operator. 
    In \cite{Mante}, the author studies the fractional Bernstein operator $(1-B_n)^\alpha$,
    $\alpha>0$, which differs completely from ours. 
    Another different formulation of the fractional Bernstein operator 
    from ours can be found in \cite{Yu}. 
    Our fractional Bernstein operator $B_{\alpha, n}$ appears to be new, being directly derived from the generalized binomial theorem.

\subsection{Notation}

We use the following notation.
\begin{itemize}
    \item
        We set $\N=\{1, 2, 3, \dots\}$ and $\N_0=\{0, 1, 2, 3, \dots\}$.
    \item
        $\mathrm{Re}(\omega)$ denotes the real part of 
        a complex number $\omega$. 
    \item 
        The indicator function of a set $A$ is denoted by $\bm{1}_A$. 
    \item 
        The cardinality of a set $A$ is denoted by $\#(A)$.
    \item 
        For $x, y \in \R$, we write 
        $x \vee y=\max\{x, y\}$ and $x \wedge y=\min\{x, y\}$. 
    \item 
        For $x \in \R$, $\lfloor x \rfloor$ and $\lceil x \rceil$ denote
        the greatest integer less than or equal to $x$, 
        and the smallest integer greater than or equal to $x$, respectively. 
    \item 
        For $a \in \R$, $\mathrm{sgn}\,a$ denotes the sign of $a$, that is, 
            \[
            \mathrm{sgn}\,a=\begin{cases}
            1 & \text{if }a>0, \\
            0 & \text{if }a=0, \\
            -1 & \text{if }a<0.
            \end{cases}
            \]
    \item 
        Let $f(n)$ and $g(n)$ be functions defined on $\N$. 
        Then, $f(n)=o(g(n))$ and $f(n)=O(g(n))$ as $n \to \infty$
        mean the usual little-$o$ and big-$O$ notations, respectively. 
        $f(n)=\Omega(g(n))$ means $g(n)=O(f(n))$. 
        Moreover, we use $f(n) \sim g(n)$ as $n \to \infty$ to indicate
        $f(n)/g(n) \to 1$ as $n \to \infty$. 
    \item 
        Let $X$ and $Y$ be two Banach spaces. 
        The domain and range of a linear operator $L \colon X \to Y$ are denoted by
        $\mathrm{Dom}(L) \subset X$ and $\mathrm{Range}(L) \subset Y$, respectively. 
    \item 
        A positive constant whose concrete expression is not important 
        is denoted by $C_{i, j}$ when it is the $j$-th positive constant in Section $i$ . 
\end{itemize}

\section{{\bf The fractional binomial distribution and its properties}}
\label{Sect:fractional-binomial-distribution}

\subsection{A fractional-order Taylor series}
\label{Subsect:fractional-order-Taylor-series}

Let $D=\{z \in \mathbb{C} \, : \, |z|<1\}$ be the unit open disc 
and $\ol{D}$ its closure in $\mathbb{C}$. 
Consider a continuous function $f$ on $\ol{D}$ such that $f$ is holomorphic in $D$. 
For each $\xi \in \R$, we define 
    \[
    f^{\#}(\xi)
    :=\int_{-1/2}^{1/2}f(e^{2\pi i x}) e^{-2\pi i x \xi} \, \dd x
    =\frac{1}{2\pi i}\int_{C}\frac{f(z)}{z^{\xi+1}} \, \dd z,
    \]
where $C$ denotes the oriented contour 
$(-1, 1) \ni t \longmapsto e^{i\pi t} \in \mathbb{C}$. 

The following is the fractional-order Taylor series with residual terms. 
Recall that the set $K_\alpha$ was defined in \eqref{Eq:definition-K_alpha}.

\begin{pr}[cf.~{\cite[Theorem~3.1]{HH}}]
\label{Prop:Fractional-Taylor-series-HH}
    Suppose that $\alpha>0$ and $\lambda \in (0, 1)$. 
    If $\alpha/2 \notin \N$, then we have 
    \begin{align}
        \alpha\sum_{j=0}^\infty f^\#(\alpha j)\lambda^{\alpha j}
        &= \sum_{\omega \in K_\alpha}f(\lambda \omega)
        -\frac{\alpha \lambda^\alpha \sin \alpha\pi}{\pi}\int_0^1 
        \frac{t^{\alpha-1}f(-t)}{|t^\alpha-\lambda^\alpha e^{-i\alpha \pi}|^2} \, \dd t,  
        \label{Eq:Fractional-Taylor-series-1}\\
        \alpha\sum_{j=-\infty}^{-1} f^\#(\alpha j)\lambda^{-\alpha j}
        &=\frac{\alpha\lambda^\alpha \sin \alpha \pi}{\pi}
        \int_0^1 
        \frac{t^{\alpha-1}f(-t)}{|e^{-i\alpha \pi}-(\lambda t)^\alpha|^2} \, \dd t. 
        \label{Eq:Fractional-Taylor-series-2}
    \end{align}
    If $\alpha/2 \in \N$, \eqref{Eq:Fractional-Taylor-series-1} and 
    \eqref{Eq:Fractional-Taylor-series-2} still hold by regarding 
    the second term of the right-hand side of \eqref{Eq:Fractional-Taylor-series-1} 
    and the right-hand side of \eqref{Eq:Fractional-Taylor-series-2} as $0$. 
\end{pr}

\noindent
We note that these two equalities in Proposition~\ref{Prop:Fractional-Taylor-series-HH}
turn out to be true when $\lambda=1$ through the limit as $\lambda \nearrow 1$,
provided that 
    \[
    \sum_{j=-\infty}^\infty |f^\#(\alpha j)|<\infty.
    \]
Since it holds that 
    \[
    f^\#(\xi)
    =\begin{cases}
    \displaystyle \frac{1}{\xi!}\frac{\dd^\xi f}{\dd z^\xi}(0) & \text{if }\xi \ge 0, \\
    0 & \text{if }\xi<0, 
    \end{cases} \qquad \xi \in \Z, 
    \]
the identity \eqref{Eq:Fractional-Taylor-series-1} is identical to
the usual Taylor series expansion of $f$ at $z=0$ when $\alpha=1$. 
On the other hand, if $\xi \notin \Z$, we see that $\Gamma(\xi+1)f^\#(\xi)$ 
is regarded as the $\xi$-order fractional derivative of $f$ at $z=0$. 
We also recall the following result. 

\begin{pr}[cf.~{\cite[Proposition~2.4]{HH}}]
\label{Prop:Property-f-sharp}
    Let $T>0$ and define $f(z)=(1+z)^T$ on $\ol{D}$. 
    Then, we have 
        \[
        f^\#(\xi)=\binom{T}{\xi}, \qquad \xi \in \R.
        \]
\end{pr}

\subsection{Moments of the fractional binomial distribution}
\label{Subsect:moments}

Let $\alpha > 0$, $n \in \N$, and $x \in (0, 1)$. 
We now derive the explicit expressions for the moments 
of the fractional binomial distribution $\mu_{\alpha, x}^{(n)}$ 
by applying Proposition~\ref{Prop:Fractional-Taylor-series-HH} to a certain holomorphic function. 
For $c \in \R$, we define a differential operator $\mathscr{D}_c$ 
acting on holomorphic functions $f$ on a domain of $\mathbb{C}$ by 
    \[
    \mathscr{D}_cf(z):=zf'(z)+cf(z).
    \]
We also define
\begin{equation}\label{Eq:definition-function-h_c^m}
    h_c^{(m)}(z)
    :=\mathscr{D}_c^m\{(1+z)^{\alpha n}\}, 
    \qquad z \in 
    \begin{cases}
    \mathbb{C} \setminus \{x \in \R \mid x \le -1\} & \text{if }\alpha n \notin \N, \\
    \mathbb{C} & \text{if }\alpha n \in \N
    \end{cases}
\end{equation}
for $m \in \N_0$. Moreover, we put
    \[
    (s)_0=1, \qquad 
    (s)_k=s(s-1) \cdots (s-k+1), \qquad k \in \N
    \]
for $s>0$. 
Then, we have the following explicit representation of $h_c^{(m)}$. 

\begin{lm}\label{Lem:Representation-h_c^m}
    For $c \in \R$, we define a sequence 
    $\{W_c(m, k)\}_{0 \le k \le m}$,  $m \in \N_0$,  by 
    \begin{equation}\label{Eq:Recurssive-formula-W_c}
        \begin{aligned}
            W_c(m, 0)&=c^m, \\
            W_c(m, k)&=(k+c)W_c(m-1, k)+W_c(m-1, k-1), \qquad k=1, 2, \dots, m-1, \\
            W_c(m, m)&=1. \\
        \end{aligned}
    \end{equation}
    Then, it holds that 
    \begin{equation}\label{Eq:Representation-h_c^m}
        h_c^{(m)}(z)
        =\sum_{k=0}^m W_c(m, k)(\alpha n)_k z^k (1+z)^{\alpha n - k} 
    \end{equation}
    for $z \in \mathbb{C} \setminus \{x \in \R \mid x \le -1\}$ if $\alpha n \notin \N$ 
    and $z \in \mathbb{C}$ if $\alpha n \in \N$.
    Moreover, each $W_c(m, k)$ is a polynomial in $c$ of degree $m-k$ with integer coefficients. 
\end{lm}
See Table~\ref{Table:Whitney} for concrete expressions of some of $W_c(m, k)$.
\begin{table}[t]
 \centering
  \caption{Expressions of $W_c(m, k)$ for $m=0, 1, 2, 3, 4$.}
  \begin{tabular}{c|c c c c c}
   \diagbox{$m$}{$k$} & 0 & 1 & 2 & 3 & 4 \\ \hline
   0 & 1 & & & & \\
   1 & $c$ & 1 & & & \\
   2 & $c^2$ & $2c+1$ & 1 & & \\
   3 & $c^3$ & $3c^2+3c+1$ & $3c+3$ & 1 & \\
   4 & $c^4$ & $4c^3+6c^2+4c+1$ & $6c^2+12c+7$ & $4c+6$ & 1  \\
  \end{tabular} 
 \label{Table:Whitney}
\end{table}

\begin{proof}
    The proof is by induction. 
    When $m=0$, we have $h_c^{(0)}(z)=(1+z)^{\alpha n}$
    and \eqref{Eq:Representation-h_c^m} holds.
    Suppose that \eqref{Eq:Representation-h_c^m} holds
    for some $m \in \N_0$. 
    Then, we have 
    \begin{align*}
        h_c^{(m+1)}(z) 
        &= \mathscr{D}_c h_c^{(m)}(z) \\
        &= z (h_c^{(m)})'(z)+ch_c^{(m)}(z) \\
        &= z\sum_{k=0}^m W_c(m, k)(\alpha n)_k 
        \{kz^{k-1}(1+z)^{\alpha n-k} + (\alpha n-k)z^k (1+z)^{\alpha n - k-1}\} \\
        &\quad+c\sum_{k=0}^m W_c(m, k)(\alpha n)_k z^k (1+z)^{\alpha n - k} \\
        &=\sum_{k=1}^m \{(k+c)W_c(m, k)+W_c(m, k-1)\}(\alpha n)_k z^{k}(1+z)^{\alpha n-k}\\
        &\quad+W_c(m, m)(\alpha n)_{m+1} z^{m+1}(1+z)^{\alpha n-m-1}
        +cW_c(m, 0)(1+z)^{\alpha n}.
    \end{align*}
    Therefore, it follows from \eqref{Eq:Recurssive-formula-W_c} that the right-hand side coincides with 
    \[
    \sum_{k=0}^{m+1}W_c(m+1, k)(\alpha n)_kz^k(1+z)^{\alpha n -k},
    \]
    which implies that \eqref{Eq:Representation-h_c^m} is also true for the case of $m+1$. 
    The latter assertion can also be shown immediately 
    by induction, noting the recurrence formula \eqref{Eq:Recurssive-formula-W_c}. 
\end{proof}

\begin{re}\normalfont
    Let $c=0$. Then, the recurrence formula \eqref{Eq:Recurssive-formula-W_c}
    gives the {\it Stirling number of the second kind} $W_0(m, k)=S(m, k)$, 
    which is the number of ways to partition a set of $m$ objects into $k$ non-empty subsets. 
    Therefore, the sequence $\{W_c(m, k)\}$ is regarded as an extension of the 
    Stirling numbers of the second kind, which in combinatorics are referred to as the 
    ($c$-){\it Whitney numbers of the second kind}. 
    To the best of our knowledge, this notion was first introduced in \cite{Mezo}, 
    and indeed the theorem below tells us that $W_c(m, k)$ 
    appears when we express the moments of the fractional binomial distribution. 
\end{re}

By applying Lemma~\ref{Lem:Representation-h_c^m}, 
we obtain the following key identity to deduce the moments
of the fractional binomial distribution $\mu_{\alpha, x}^{(n)}$
in Definition~\ref{Def:fractional-binomial-distribution}. 

\begin{tm}
\label{Thm:Moment-fractional-binomial-distribution}
    Let $\alpha>0$, $n \in \N$, $x \in (0, 1)$,  $c \in \R$, and $m \in \N_0$. 
    If $\alpha n \notin \N$, suppose further that $m \le \alpha n$. 
    Then, we have 
    \begin{align}
        &Z_{\alpha, x}^{(n)}
        \sum_{j=0}^n (\alpha j + c)^m \mu_{\alpha, x}^{(n)}(\{j\}) \nn \\
        &=\sum_{k=0}^m W_c(m, k)(\alpha n)_k x^k
        +\sum_{k=0}^m W_c(m, k)(\alpha n)_k 
        \left\{\sum_{\omega \in K_\alpha \setminus\{1\}}
        (x\omega)^k(1-x+x\omega)^{\alpha n-k}\right\}  \nn \\
        &\quad-\frac{\alpha x^\alpha
        (1-x)^{\alpha(n+1)}\sin \alpha \pi}{\pi} 
        \int_0^1 \frac{t^{\alpha-1}h_c^{(m)}(-t)}{|(t(1-x))^\alpha - x^\alpha e^{-i \alpha \pi}|^2} \, \dd t \nn \\
        &\quad-(-1)^m \frac{\alpha x^{\alpha(n+1)}(1-x)^\alpha \sin \alpha \pi}{\pi}
        \int_0^1 \frac{t^{\alpha-1}h_{-(c+\alpha n)}^{(m)}(-t)}{|(1-x)^{\alpha}e^{-i \alpha \pi}-(tx)^\alpha|^2} \, \dd t.
    \label{Eq:Moment-fractional-binomial-distribution}
    \end{align}
    Here, if $\alpha \in \N$, then the third and fourth terms of the right-hand side are regarded as $0$.
\end{tm}

\begin{proof}
    The proof is split into two parts. 

    \smallskip
    \noindent
    {\bf Step 1.}
    Let $f$ be a continuous function on $\ol{D}$ such that $f$ is holomorphic on $D$ and $f(-1)=0$. 
    Suppose that a function $g$ defined by 
        \[
        g(z):=\mathscr{D}_cf(z)=zf'(z)+cf(z), \qquad z \in D
        \] 
    is extended as a continuous function on $\ol{D}$ and 
    the function $(-1/2, 1/2) \ni x \longmapsto f(e^{2\pi i x})$ is of class $C^1$. 
    Then, we have 
    \begin{align}
        g^\#(\xi)
        &= \int_{-1/2}^{1/2} e^{2\pi i x}f'(e^{2\pi i x})
        e^{-2\pi i x \xi} \, \dd x + cf^\#(\xi) \nn \\
        &= \lim_{\ve \searrow 0}\frac{1}{2\pi i}\int_{-1/2+\ve}^{1/2-\ve} 
        \left(\frac{\dd}{\dd x}f(e^{2\pi i x})\right)
        e^{-2\pi i x \xi} \, \dd x+ cf^\#(\xi) \nn \\
        &=\lim_{\ve \searrow 0}\left(\left[\frac{1}{2\pi i}f(e^{2\pi i x})e^{-2\pi i x \xi}\right]_{x=-1/2+\ve}^{x=1/2-\ve}
        +\int_{-1/2+\ve}^{1/2-\ve}f(e^{2\pi i x})\xi e^{-2\pi i x \xi} \, \dd x\right)
        + cf^\#(\xi) \nn \\
        &=(\xi+c)f^\#(\xi), \qquad \xi \in \R. 
        \label{Eq:g-sharp}
    \end{align}
    In the last equality, we used $f(-1)=0$.
    Because $m \le \alpha n$ or $\alpha n \in \N$ holds, we can take the function   
    $h_c^{(k)}$, $k=0, 1, \dots, m-1$ defined by \eqref{Eq:definition-function-h_c^m} as $f$. 
    Therefore, we obtain
    \begin{align}\label{Eq:h_c^m-sharp}
        (h_c^{(m)})^\#(\xi)
        =(\xi+c)(h_c^{(m-1)})^\#(\xi) 
        =\cdots 
        =(\xi+c)^m(h_c^{(0)})^\#(\xi)
        =(\xi+c)^m \binom{\alpha n}{\xi}
    \end{align}
    for $\xi \in \R$, from \eqref{Eq:g-sharp} and Proposition~\ref{Prop:Property-f-sharp}.

    \smallskip
    \noindent{\bf Step 2.} 
    We apply Proposition~\ref{Prop:Fractional-Taylor-series-HH} 
    to the function $f(z)=h_c^{(m)}(z)$. 
    Then, it follows from \eqref{Eq:h_c^m-sharp} that
    \begin{align}
        \alpha \sum_{j=0}^\infty (\alpha j+c)^m \binom{\alpha n}{\alpha j}\lambda^{\alpha j}
        &= \sum_{\omega \in K_\alpha}h_c^{(m)}(\lambda \omega)-\frac{\alpha \lambda^\alpha \sin \alpha \pi}{\pi} 
        \int_0^1 \frac{t^{\alpha-1}h_c^{(m)}(-t)}{|t^\alpha - \lambda^\alpha e^{-i \alpha \pi}|^2} \, \dd t, 
        \label{Eq:Taylor-series-h_c^m-1} \\
        \alpha \sum_{j=-\infty}^{-1} (\alpha j+c)^m \binom{\alpha n}{\alpha j}\lambda^{-\alpha j}
        &= \frac{\alpha \lambda^\alpha \sin \alpha \pi}{\pi}
        \int_0^1 \frac{t^{\alpha-1}h_c^{(m)}(-t)}{|e^{-i \alpha \pi}-(\lambda t)^\alpha|^2} \, \dd t.
        \label{Eq:Taylor-series-h_c^m-2}
    \end{align}
    Here, if $\alpha \in \N$, then the second term of the right-hand side of 
    \eqref{Eq:Taylor-series-h_c^m-1} and the right-hand side of \eqref{Eq:Taylor-series-h_c^m-2} 
    are regarded as 0. Because it holds that  
    \begin{align*}
      \alpha \sum_{j=n+1}^\infty (\alpha j+c)^m \binom{\alpha n}{\alpha j}\lambda^{\alpha j} 
      &= \alpha \sum_{l=-\infty}^{-1} \{\alpha (n-l)+c\}^m 
      \binom{\alpha n}{\alpha (n-l)}\lambda^{\alpha (n-l)}\\
      &=(-1)^m\lambda^{\alpha n}  \alpha \sum_{l=-\infty}^{-1}
      \{\alpha l-(c+\alpha n)\}^m \binom{\alpha n}{\alpha l}\lambda^{-\alpha l} \\
      &=(-1)^m \lambda^{\alpha n}  
      \frac{\alpha \lambda^\alpha \sin \alpha \pi}{\pi}
      \int_0^1 \frac{t^{\alpha-1}h_{-(c+\alpha n)}^{(m)}(-t)}{|e^{-i \alpha \pi}-(\lambda t)^\alpha|^2} \, \dd t
    \end{align*}
    by using \eqref{Eq:Taylor-series-h_c^m-2} with $c$ replaced by $-(c+\alpha n)$, 
    we apply it to \eqref{Eq:Taylor-series-h_c^m-1} to obtain
    \begin{align}
      \alpha \sum_{j=0}^n (\alpha j+c)^m \binom{\alpha n}{\alpha j}\lambda^{\alpha j} 
      &= \alpha \sum_{j=0}^\infty (\alpha j+c)^m \binom{\alpha n}{\alpha j}
      \lambda^{\alpha j}
        -\alpha \sum_{j=n+1}^\infty (\alpha j+c)^m \binom{\alpha n}{\alpha j}\lambda^{\alpha j} \nn \\
      &=\sum_{\omega \in K_\alpha}h_c^{(m)}(\lambda\omega)-\frac{\alpha \lambda^\alpha \sin \alpha \pi}{\pi} 
        \int_0^1 \frac{t^{\alpha-1}h_c^{(m)}(-t)}{|t^\alpha - \lambda^\alpha e^{-i \alpha \pi}|^2} \, \dd t \nn \\
      &\quad-(-1)^m\frac{\alpha \lambda^{\alpha(n+1)} \sin \alpha \pi}{\pi}
        \int_0^1 \frac{t^{\alpha-1}h_{-(c+\alpha n)}^{(m)}(-t)}{|e^{-i \alpha \pi}
        -(\lambda t)^\alpha|^2} \, \dd t.
        \label{Eq:Sum-Taylor-series}
    \end{align}
    We put $\lambda=x/(1-x)$, $0<x<1/2$ 
    and multiply both sides in \eqref{Eq:Sum-Taylor-series} by $(1-x)^{\alpha n}$. 
    Then, we obtain
    \begin{align}
      &\alpha \sum_{j=0}^n (\alpha j+c)^m \binom{\alpha n}{\alpha j}x^{\alpha j}(1-x)^{\alpha(n-j)} \nn \\
      &= (1-x)^{\alpha n}\sum_{\omega \in K_\alpha}
      h_c^{(m)}\left(\frac{x\omega}{1-x}\right) \nn \\
      &\quad-\frac{\alpha x^\alpha
      (1-x)^{\alpha(n+1)}\sin \alpha \pi}{\pi} 
      \int_0^1 \frac{t^{\alpha-1}h_c^{(m)}(-t)}{|(t(1-x))^\alpha - x^\alpha e^{-i \alpha \pi}|^2} \, \dd t \nn \\
      &\quad-(-1)^m \frac{\alpha x^{\alpha(n+1)}(1-x)^\alpha \sin \alpha \pi}{\pi}
      \int_0^1 \frac{t^{\alpha-1}h_{-(c+\alpha n)}^{(m)}(-t)}{|(1-x)^{\alpha}e^{-i \alpha \pi}-(tx)^\alpha|^2} \, \dd t.
      \label{Eq:Sum-Taylor-series-x}
    \end{align}
Because 
    \[
    |(t(1-x))^\alpha - x^\alpha e^{-i\alpha \pi}|^2
    =\{(t(1-x))^\alpha - x^\alpha e^{-i\alpha \pi}\}
    \{(t(1-x))^\alpha - x^\alpha e^{i\alpha \pi}\}
    \]
and 
    \[
    |(1-x)^\alpha e^{-i\alpha \pi} - (tx)^\alpha|^2
    =\{(1-x)^\alpha e^{-i\alpha \pi} - (tx)^\alpha\}
    \{(1-x)^\alpha e^{i\alpha \pi} - (tx)^\alpha\}
    \]
are extended to holomorphic functions in $x$ on a complex domain containing 
the set $\{x \in \R \mid 0<x<1\}$ and do not vanish if $\alpha \notin \N$,
\eqref{Eq:Sum-Taylor-series-x} also holds for $0<x<1$ by the identity theorem. 
Moreover, by applying \eqref{Eq:Representation-h_c^m} with $z=x\omega/(1-x)$ and 
$\omega \in K_\alpha$ to the first term of the right-hand side of \eqref{Eq:Sum-Taylor-series-x}, 
we have 
    \begin{align*}
    &(1-x)^{\alpha n}\sum_{\omega \in K_\alpha} h_c^{(m)}\left(\frac{x\omega}{1-x}\right) \\
    &=\sum_{k=0}^m W_c(m, k) (\alpha n)_k \left\{\sum_{\omega \in K_\alpha}(x\omega)^k 
    (1-x+x\omega)^{\alpha n-k}\right\}\\
    &=\sum_{k=0}^m W_c(m, k) (\alpha n)_kx^k+\sum_{k=0}^m W_c(m, k) (\alpha n)_k 
    \left\{\sum_{\omega \in K_\alpha \setminus \{1\}}(x\omega)^k (1-x+x\omega)^{\alpha n-k}\right\}.
    \end{align*}
We remark that if $\omega=-1 \in K_\alpha$, then $\alpha/2 \in \N$ so that 
$h_c^{(m)}(x\omega/(1-x))$ is well-defined.
Thus, \eqref{Eq:Sum-Taylor-series-x} implies \eqref{Eq:Moment-fractional-binomial-distribution}. 
\end{proof}

\begin{pr}\label{Prop:Order-moment-fractional-binomial}
\begin{enumerate}
\item As $n \to \infty$,
\begin{equation}\label{Eq:Order-normalized-constant}
    Z_{\alpha, x}^{(n)}=1+O(e^{-\delta n}),
    \qquad \frac{1}{Z_{\alpha, x}^{(n)}}=1+O(e^{-\delta n}).
\end{equation}
\item For each $m \in \N_0$,  
    \begin{equation}\label{Eq:Order-moment-fractional-binomial}
        \alpha \sum_{j=0}^n (\alpha j+c)^m 
        \binom{\alpha n}{\alpha j}x^{\alpha j}(1-x)^{\alpha(n-j)}
        =\sum_{k=0}^m W_c(m, k)(\alpha n)_kx^k + (|c|^m+1)O(e^{-\delta n})
    \end{equation}
    as $n \to \infty$.
\end{enumerate}
Here, the symbol $O(e^{-\delta n})$ means that there exist positive constants $C$ and $\delta$ depending only on $x$, $\alpha$ (and $m$) that are chosen uniformly on any compact subset of $x \in (0, 1)$ and any compact subset of $\alpha \in (0, \infty)$ such that the modulus of the corresponding term is dominated by $C e^{-\delta n}$. 
\end{pr}
In what follows, we use the notation $O(e^{-\delta n})$ in this sense; $\delta$ may change from line to line. 

In order to prove Proposition~\ref{Prop:Order-moment-fractional-binomial}, 
we need to give the following estimates of the third term 
on the right-hand side of \eqref{Eq:definition-normalized-constant}.

\begin{lm}\label{Lem:Estimates-integral-terms}
    Let $\alpha>0$, $n \in \N$, and $x \in (0, 1)$. Then, we have 
    \begin{align}
        \Bigg|\frac{\alpha x^\alpha (1-x)^{\alpha(n+1)} \sin \alpha \pi}{\pi}
        \int_0^1 
        \frac{t^{\alpha-1}(1-t)^{\alpha n}}{|(t(1-x))^\alpha-x^\alpha e^{-i\alpha \pi}|^2}
        \, \dd t\Bigg| 
        &\le (1-x)^{\alpha n}\psi(\alpha)
        \label{Eq:Integral-estimate-1}
        \intertext{and}
        \Bigg|\frac{\alpha x^{\alpha(n+1)} (1-x)^\alpha \sin \alpha \pi}{\pi}
        \int_0^1 
        \frac{t^{\alpha-1}(1-t)^{\alpha n}}{|(1-x)^\alpha e^{-i\alpha \pi}-(tx)^\alpha|^2}
        \, \dd t\Bigg| 
        &\le x^{\alpha n}\psi(\alpha), 
        \label{Eq:Integral-estimate-2}
    \end{align}
    where $\psi(\alpha)$ is the periodic function of period $2$ such that
    $\psi(\alpha)=1-|\alpha|$ for $|\alpha| \le 1$. 
\end{lm}

\begin{proof}
    Because the left-hand sides of \eqref{Eq:Integral-estimate-1} and 
    \eqref{Eq:Integral-estimate-2} are regarded as $0$ if $\alpha \in \N$, 
    we may suppose that $\alpha$ is not an integer. 
    It is sufficient to show \eqref{Eq:Integral-estimate-1}
    because \eqref{Eq:Integral-estimate-2} is immediately obtained 
    by replacing $x$ in \eqref{Eq:Integral-estimate-1} with $1-x$. 
    We have
    \begin{align}
        &\Bigg|\frac{\alpha x^\alpha (1-x)^{\alpha(n+1)} \sin \alpha \pi}{\pi}
        \int_0^1 
        \frac{t^{\alpha-1}(1-t)^{\alpha n}}{|(t(1-x))^\alpha-x^\alpha e^{-i\alpha \pi}|^2}
        \, \dd t\Bigg|  \nn \\
        &\le \frac{\alpha x^\alpha (1-x)^{\alpha(n+1)}}{\pi}|\sin \alpha \pi|
        \int_0^1 \frac{t^{\alpha-1}}{|(t(1-x))^\alpha-x^\alpha e^{-i\alpha \pi}|^2} \, \dd t.
        \label{Eq:integral-term-bound-1}
    \end{align} 
    By the change of variable $s=((1-x)t/x)^\alpha$
    and \cite[Equation~12, page 329]{GR}, 
    the right-hand side of \eqref{Eq:integral-term-bound-1} becomes
    \begin{align*}
        &\frac{(1-x)^{\alpha n}}{\pi}|\sin \alpha \pi|
        \int_0^{(\frac{1-x}{x})^\alpha} \frac{\dd s}{|s-e^{-i\alpha \pi}|^2} \nn \\
        &\le \frac{(1-x)^{\alpha n}}{\pi}|\sin \alpha \pi|
        \int_0^\infty \frac{\dd s}{(s-\cos \alpha \pi)^2+\sin^2\alpha \pi} \nn \\
        &=\frac{(1-x)^{\alpha n}}{\pi}|\sin \alpha \pi|
        \times \frac{\pi \psi(\alpha)}{|\sin \alpha \pi|}
        =(1-x)^{\alpha n}\psi(\alpha),
    \end{align*}
which is the desired estimate. 
\end{proof}

We are now in a position to prove Proposition~\ref{Prop:Order-moment-fractional-binomial}.

\begin{proof}[Proof of Proposition~{\rm \ref{Prop:Order-moment-fractional-binomial}}]
For $\omega \in K_\alpha \setminus \{1\}$, that is, 
$\omega = e^{i\theta}$ for some $\theta \in (-\pi, \pi] \setminus \{0\}$, 
we observe that 
    \begin{align}
        |1-x+x\omega|^2 &= (1-x)^2+x^2+2x(1-x)\mathrm{Re}(\omega) \nn \\
        &=1-2x(1-x)(1-\mathrm{Re}(\omega))<1, 
        \label{Eq:Estimate-1-x+xomega}
    \end{align}
which implies that for each $k \in \N_0$, 
    \begin{equation}\label{Eq:Order-1-x+xomega}
        \sum_{\omega \in K_\alpha \setminus\{1\}}
        (x\omega)^k(1-x+x\omega)^{\alpha n-k}=O(e^{-\delta n})
    \end{equation}
as $n \to \infty$ for some $\delta>0$. 
\eqref{Eq:Order-normalized-constant} follows from \eqref{Eq:definition-normalized-constant}, \eqref{Eq:Order-1-x+xomega} and Lemma~\ref{Lem:Estimates-integral-terms}.
Furthermore, Lemma~\ref{Lem:Representation-h_c^m} implies that 
    \begin{equation}\label{Eq:Estimate-W_c}
        |W_c(m, k)| \le C_{2, 1}(|c|^{m-k}+1), \qquad m \in \N_0, \,\, k=0, 1, \dots, m
    \end{equation}
for some $C_{2, 1}>0$ depending only on $m \in \N_0$.
Then, for large $n$ such that $\alpha n \ge m$, we have
    \begin{align}
        |h_c^{(m)}(-t)| 
        &\le \sum_{k=0}^m C_{2,2}(|c|^{m-k}+1)n^k,
        \label{Eq:Estimate-h_c^m-1} \\
        |h_{-(c+\alpha n)}^{(m)}(-t)| 
        &\le \sum_{k=0}^m C_{2, 3}(|c+\alpha n|^{m-k}+1)n^k
        \label{Eq:Estimate-h_c^m-2}
    \end{align}
for $t \in (0, 1)$. 
Hence, the second term on the right-hand side of \eqref{Eq:Moment-fractional-binomial-distribution} is estimated as 
    \begin{align}
        &\left|\sum_{k=0}^m W_c(m, k)(\alpha n)_k\left\{\sum_{\omega \in K_\alpha \setminus \{1\}}(x\omega)^k(1-x+x\omega)^{\alpha n-k}\right\}\right| \nn \\
        &\le \sum_{k=0}^m C_{2, 4}(|c|^{m-k}+1)n^k \left|
        \sum_{\omega \in K_\alpha \setminus \{1\}}(x\omega)^k(1-x+x\omega)^{\alpha n-k}\right|
        =(|c|^m+1)O(e^{-\delta n}) \nn
    \end{align}
as $n \to \infty$ for some $\delta>0$, 
by using \eqref{Eq:Order-1-x+xomega} and \eqref{Eq:Estimate-W_c}. 
Moreover, \eqref{Eq:Estimate-h_c^m-1}, \eqref{Eq:Estimate-h_c^m-2} 
and Lemma~\ref{Lem:Estimates-integral-terms}
imply that the modulus of the sum of the third and fourth terms on the right-hand side of 
\eqref{Eq:Moment-fractional-binomial-distribution} is dominated by
    \begin{align}
        &(1-x)^{\alpha n}\psi(\alpha)\sum_{k=0}^m C_{2,2}(|c|^{m-k}+1)n^k
        +x^{\alpha n}\psi(\alpha)\sum_{k=0}^m C_{2,3}(|c+\alpha n|^{m-k}+1)n^k, \nn
    \end{align}
which is $(|c|^m+1)O(e^{-\delta n})$ as well. 
Thus, \eqref{Eq:Order-moment-fractional-binomial} holds.
Note that these items  are regarded as 0 if $\alpha \in \N$.
\end{proof}

We often use Proposition~\ref{Prop:Order-moment-fractional-binomial}~(2) with 
$c=-\alpha nx$ in the rest of this paper. 
When $c=-\alpha nx$, the second term
on the right-hand side of \eqref{Eq:Order-moment-fractional-binomial} is 
$O(e^{-\delta n})$ with possibly different $\delta>0$. 
A direct calculation also gives
    \begin{equation}\label{Eq:Examples-leading-terms}
        \sum_{k=0}^m W_{-\alpha nx}(m, k)(\alpha n)_kx^k
        =\begin{cases}
        0 & \text{if }m=1, \\
        \alpha n x(1-x) & \text{if }m=2, \\
        \alpha n x(1-x)(1-2x) & \text{if }m=3, \\
        \alpha n x(1-x)\{3(\alpha n-2)x(1-x)+1\} & \text{if }m=4.
        \end{cases} 
    \end{equation}

The mean and variance of $\mu_{\alpha, x}^{(n)}$ are obtained 
as follows.

\begin{co}\label{Cor:Mean-variance-fractional-binomial}
    Let $S_{\alpha, x}^{(n)}$ be a random variable whose probability law is $\mu_{\alpha, x}^{(n)}$. 
    Then, for some $\delta>0$, we have 
    \begin{align}
        \E[S_{\alpha, x}^{(n)}] 
        = \sum_{j=0}^n j \mu_{\alpha, x}^{(n)}(\{j\}) 
        = nx+O(e^{-\delta n})
        \label{Eq:Mean-fractional-binomial}
    \end{align}
    and 
    \begin{align*}
        \mathrm{Var}(S_{\alpha, x}^{(n)})
        =\sum_{j=0}^n \left(j-\E[S_{\alpha, x}^{(n)}]\right)^2\mu_{\alpha, x}^{(n)}(\{j\})
        =\frac{1}{\alpha}nx(1-x)+O(e^{-\delta n})
    \end{align*}
    as $n \to \infty$. 
\end{co}

\begin{proof}
    Applying Proposition \ref{Prop:Order-moment-fractional-binomial}~(2) with $c=0$ and $m=1$, we have 
    \begin{align*}
        \alpha \sum_{j=0}^n (\alpha j) \binom{\alpha n}{\alpha j}x^{\alpha j}
        (1-x)^{\alpha(n-j)}&=W_0(1, 0)+W_0(1, 1)\alpha nx+O(e^{-\delta n}) \\
        &=\alpha nx+O(e^{-\delta n})
    \end{align*}
    as $n \to \infty$. This and Proposition~\ref{Prop:Order-moment-fractional-binomial}~(1) imply \eqref{Eq:Mean-fractional-binomial}.
    We also apply \eqref{Eq:Examples-leading-terms} and Proposition~\ref{Prop:Order-moment-fractional-binomial}~(2)
    with $c=-\alpha nx$ and $m=2$ to obtain
    \[
    \sum_{j=0}^n (j-nx)^2 \mu_{\alpha, x}^{(n)}(\{j\})
    =\frac{1}{Z_{\alpha, x}^{(n)}}\left\{\frac{1}{\alpha}nx(1-x)+O(e^{-\delta n})\right\}=\frac{1}{\alpha}nx(1-x)+O(e^{-\delta n})
    \]
    as $n \to \infty$. Hence, \eqref{Eq:Mean-fractional-binomial} leads to
    \begin{align*}
        \mathrm{Var}(S_{\alpha, x}^{(n)})
        &=\sum_{j=0}^n (j-nx)^2 \mu_{\alpha, x}^{(n)}(\{j\}) 
        - \left(nx-\E[S_{\alpha, x}^{(n)}]\right)^2 \\
        &=\frac{1}{\alpha}nx(1-x)+O(e^{-\delta n})
    \end{align*}
    as $n \to \infty$. 
\end{proof}

\subsection{Uniform bound of $Z_{\alpha, x}^{(n)}$}
\label{Subsect:uniform bound}

Let $\alpha>0$, $n \in \N$, and $x \in [0, 1]$. 
We establish a uniform bound of the normalized constant $Z_{\alpha, x}^{(n)}$
with respect to $x$ and $n$, and again we note that $Z_{1, x}^{(n)}=1$. 

\begin{tm}[Uniform bound of $Z_{\alpha, x}^{(n)}$]
\label{Thm:Uniform-bound-normalized-constant}
    {\rm (1)} 
    Let $\alpha \in (0, 2)$. Then, we have 
    \begin{equation}\label{Eq:Uniform-bound-normalized-constant-1}
        \alpha \wedge 1 \le Z_{\alpha, x}^{(n)} \le \alpha \vee 1        
    \end{equation}
    for $n \in \N$ and $x \in [0, 1]$. 
    
    \smallskip
    \noindent
    {\rm (2)}
    Let $\alpha \in [2, \infty)$. Then, we have 
    \begin{equation}\label{Eq:Uniform-bound-normalized-constant-2}
        \alpha^{1-\alpha^2}
        \le Z_{\alpha, x}^{(n)} \le \alpha      
    \end{equation}
    for $n \in \N$ and $x \in [0, 1]$.
\end{tm}

\begin{proof}
     If $\alpha \in (0, 1)$, then \cite[Theorem~1.2]{HH} implies
     that $Z_{\alpha, x}^{(n)} \le 1$ for $n \in \N$ and $x \in [0, 1]$. 
     On the other hand, Equation (1.7) in \cite[Theorem~1.3]{GS} leads to
     $Z_{\alpha, x}^{(n)} \le \alpha$ for $n \in \N$ and $x \in [0, 1]$ 
     when $\alpha \ge 1$. 
     Hence, the uniform upper bounds of \eqref{Eq:Uniform-bound-normalized-constant-1} and 
     \eqref{Eq:Uniform-bound-normalized-constant-2} have been obtained. 

     We next discuss the lower bounds of $Z_{\alpha, x}^{(n)}$. 
     Because it holds that $Z_{\alpha, x}^{(n)}=Z_{\alpha, 1-x}^{(n)}$
     and $Z_{\alpha, 0}^{(n)}=\alpha$, 
     it suffices to consider only the case where $x \in (0, 1/2]$. 
     It follows from \cite[Proposition~2.6]{HH} that 
        \[
        R_\alpha^{(n)}(\lambda) \in 
        \begin{cases}
            (0, 1-\alpha) & \text{if }\alpha \in (0, 1), \\
            (1-\alpha, 0) & \text{if }\alpha \in (1, 2)
        \end{cases}
        \]
     for $\lambda \in (0, 1]$, where $R_\alpha^{(n)}(\lambda)$ is defined by 
     \begin{equation}\label{Eq:Remainder-term}
         R_\alpha^{(n)}(\lambda)
         =-\alpha\sum_{j=0}^n \binom{\alpha n}{\alpha j}\lambda^{\alpha j}+(1+\lambda)^{\alpha n}. 
     \end{equation}
     Letting $\lambda=x/(1-x)$ and multiplying both sides of \eqref{Eq:Remainder-term} 
     by $(1-x)^{\alpha n}$, we have 
        \[
            -\alpha\sum_{j=0}^n \binom{\alpha n}{\alpha j}x^{\alpha j}(1-x)^{\alpha(n-j)} +1 \in 
            \begin{cases}
            \left(0, (1-\alpha)(1-x)^{\alpha n}\right) \subset (0, 1-\alpha) 
            & \text{if }\alpha \in (0, 1), \\
            \left((1-\alpha)(1-x)^{\alpha n}, 0\right)
            &\text{if }\alpha \in (1, 2).
            \end{cases}
        \]
     This implies that $Z_{\alpha, x}^{(n)} > \alpha$
     if $\alpha \in (0, 1)$ and $Z_{\alpha, x}^{(n)} > 1$ if $\alpha \in (1, 2)$. 
     Hence, all that remains is the case $\alpha \ge 2$. We have 
        \begin{align}
            Z_{\alpha, x}^{(n)}
            &\ge \alpha \binom{\alpha n}{0} (1-x)^{\alpha n} \nn \\
            &=\alpha(1-x)^{\frac{1}{x} \times \alpha n x} \nn \\
            &\ge \alpha \left(1-\frac{1}{2}\right)^{2\alpha nx}
            =\alpha \times 2^{-2\alpha nx},
            \qquad n \in \N, \, x \in (0, 1/2]
            \label{Eq:pf-bound-1}
        \end{align}
    by noting the fact that the function $(0, 1/2] \ni x \longmapsto (1-x)^{1/x} \in \R$ 
    is monotonically decreasing. 
    Let $\gamma>0$ be a constant that is specified later. 
    Then, \eqref{Eq:pf-bound-1} implies that 
        \[
        Z_{\alpha, x}^{(n)} \ge \alpha \times 2^{-2\gamma},
        \qquad n \in \N, \, x \in (0, 1/2] \,\, \text{with $\alpha nx \le \gamma$.}
        \]
    We now deal with the upper bound of 
    $\sum_{\omega \in K_\alpha \setminus \{1\}}(1-x+x\omega)^{\alpha n}$
    for $n \in \N$ and $x \in (0, 1/2]$ with $\alpha nx > \gamma$. 
    From \eqref{Eq:Estimate-1-x+xomega}, we have 
    \begin{align}
        &\sum_{\omega \in K_\alpha \setminus \{1\}}(1-x+x\omega)^{\alpha n} \nn \\
        &\le \sum_{\omega \in K_\alpha \setminus \{1\}}
        \{1-2x(1-x)(1-\mathrm{Re}(\omega))\}^{\alpha n/2} \nn \\
        &\le \sum_{\omega \in K_\alpha \setminus \{1\}}
        \left[\exp\{-2x(1-x)(1-\mathrm{Re}(\omega))\}\right]^{\alpha n/2} \nn \\
        &\le \alpha \exp\left\{-\alpha n x(1-x)\left(1-\max_{\omega \in K_\alpha \setminus \{1\}}\mathrm{Re}(\omega)\right)\right\} \nn \\
        &\le \alpha \exp\left(-\alpha nx \times \frac{1}{2} \times \frac{8}{\alpha^2}\right)
        =\alpha \exp\left(-\alpha nx \times \frac{4}{\alpha^2}\right)
        \label{Eq:pf-bound-2}
    \end{align}
    for $n \in \N$ and $x \in (0, 1/2]$, 
    where we used 
    $\#(K_\alpha)=\lfloor \alpha/2 \rfloor +\lceil \alpha/2 \rceil \le \alpha+1$
    in the fourth line and 
        \[
            \max_{\omega \in K_\alpha \setminus \{1\}}
            \mathrm{Re}(\omega)=\cos \frac{2\pi}{\alpha} \le 1-\frac{8}{\alpha^2}
        \]
    in the last line. 
    Here, the estimate $\cos(2\pi/\alpha) \le 1-8/\alpha^2$ follows from the inequality 
        \[
            \cos t \le 1 -\frac{2t^2}{\pi^2}, \qquad t \in [0, \pi],
        \]
    which is straightforward by elementary calculus. 
    Therefore, we obtain
        \begin{align*}
            Z_{\alpha, x}^{(n)}
            &\ge 1-\alpha \exp\left(-\alpha nx \times \frac{4}{\alpha^2}\right)
            -\{(1-x)^{\alpha n}+x^{\alpha n}\}\psi(\alpha) \\
            &\ge 1-\alpha \exp\left(-\alpha nx \times \frac{4}{\alpha^2}\right)
            -2e^{-\alpha nx}
            \end{align*}
    by applying \eqref{Eq:pf-bound-2} and Lemma~\ref{Lem:Estimates-integral-terms} 
    to \eqref{Eq:definition-normalized-constant}. 
    In particular, we see that 
        \[
            Z_{\alpha, x}^{(n)} \ge 
            1-\alpha \exp\left(-\frac{4\gamma}{\alpha^2}\right)-2e^{-\gamma}
        \]
    when $\alpha nx>\gamma$. Bringing everything together, it holds that 
        \[
            Z_{\alpha, x}^{(n)}
            \ge \begin{cases}
            \alpha \times 2^{-2\gamma} & \text{if }\alpha nx \le \gamma,\\
            \dis 1-\alpha \exp\left(-\frac{4\gamma}{\alpha^2}\right)-2e^{-\gamma}
            & \text{if }\alpha nx >\gamma.
            \end{cases}
        \]
    Letting $\gamma=(\alpha^2 \log \alpha)/(2\log 2)$, we have 
        \[
            \alpha \times 2^{-2\gamma}=\alpha \times 2^{-\alpha^2\log_2 \alpha}
            =\alpha^{1-\alpha^2} \le 2^{1-2^2}=0.125
        \]
    and 
        \begin{align*}
            1-\alpha \exp\left(-\frac{4\gamma}{\alpha^2}\right)-2e^{-\gamma}
            &= 1-\alpha e^{-2\log_2\alpha}-2\alpha^{-\alpha^2/(2\log 2)} \\
            &=1-\alpha^{1-2/\log 2}-2\alpha^{-\alpha^2/(2\log 2)} \\
            &\ge 1-2^{1-2/\log 2}-2 \times 2^{-2^2/(2\log 2)}\\
            &=1-4e^{-2}=0.458\cdots.
            \end{align*}
    Therefore, we obtain  
        \[
            Z_{\alpha, x}^{(n)} \ge \min
            \left\{\alpha \times 2^{-2\gamma}, 1-\alpha 
            \exp\left(-\frac{4\gamma}{\alpha^2}\right)-2e^{-\gamma}\right\}
            \Bigg|_{\gamma=(\alpha^2\log \alpha)/(2\log 2)}
            =\alpha^{1-\alpha^2}
        \]
    for $n \in \N$ and $x \in (0, 1/2]$. 
\end{proof}

\subsection{Uniform estimates for moments of the fractional binomial distribution}
\label{Subsect:uniform estimates}

In Theorem~\ref{Thm:Moment-fractional-binomial-distribution}, 
we discussed the explicit representations of moments 
of the fractional binomial distribution $\mu_{\alpha, x}^{(n)}$. 
Here, we discuss their uniform estimates with respect to $n \in \N$ and $x \in [0, 1]$. 
In order to deduce them, the following lemma is useful. 

\begin{lm}\label{Lem:pre-ineq}
    Let $p>0$ and $q>0$. Then, we have 
        \[
            \max_{t \in [0, 1]}t^p(1-t)^q \le \left(\frac{p}{eq}\right)^p. 
        \]
\end{lm}

\begin{proof}
    Let $g(t)=t^p(1-t)^q$ for $t \in [0, 1]$. 
    Then, we have 
        \[
            g'(t)=t^{p-1}(1-t)^{q-1}\{p(1-t)-qt\}.
        \]
    Therefore, we see that the function $g$ attains its maximum at $t=p/(p+q)$, which is
        \[
            g\left(\frac{p}{p+q}\right)
            =\left(\frac{p}{q}\right)^p \left\{\left(1-\frac{p}{p+q}\right)^{\frac{p+q}{p}}\right\}^p 
            \le  \left(\frac{p}{eq}\right)^p. \qedhere
        \] 
\end{proof}

The purpose of this section is to show the following. 

\begin{tm}\label{Thm:Unifrom-moment-estimate}
    Let $m \in \N$ and $K$ be a compact subset of $(0, \infty)$. 
    Then, we have 
    \begin{align}\label{Eq:Uniform-moment-estimate}
        &\sup_{\alpha \in K}\max_{x \in [0, 1]}
        \Bigg|\alpha \sum_{j=0}^n \binom{\alpha n}{\alpha j}
        x^{\alpha j}(1-x)^{\alpha(n-j)}\left(\frac{j}{n}-x\right)^m \nn \\
        &\qquad\qquad-(\alpha n)^{-m}\sum_{k=0}^m W_{-\alpha nx}(m, k)(\alpha n)_kx^k
        \Bigg|=O\left(\frac{1}{n^m}\right) 
    \end{align}
    as $n \to \infty$. 
\end{tm}

\begin{proof}
    Let us set 
        \[
            M_{\alpha, m}^{(n)}(x)
            := \alpha \sum_{j=0}^n \binom{\alpha n}{\alpha j}
            x^{\alpha j}(1-x)^{\alpha(n-j)}\left(\frac{j}{n} - x\right)^m. 
        \]
    By continuity, it suffices to estimate the supremum over $x \in (0, 1)$. 
    The proof is split into three parts.

    \smallskip
    \noindent
    {\bf Step 1}. 
    By applying Theorem~\ref{Thm:Moment-fractional-binomial-distribution} with $c=-\alpha nx$, we have 
    \begin{align}
        &(\alpha n)^m  M_{\alpha, m}^{(n)}(x) \nn \\
        &= \sum_{k=0}^m W_{-\alpha nx}(m, k)(\alpha n)_k x^k \nn \\
        &\quad+\sum_{k=0}^m W_{-\alpha nx}(m, k)(\alpha n)_k \Bigg\{ 
        \sum_{\omega \in K_\alpha \setminus \{1\}}(x \omega)^k(1-x+x\omega)^{\alpha n-k}\Bigg\} \nn \\
        &\quad-\Bigg\{\frac{\alpha x^\alpha
        (1-x)^{\alpha(n+1)}\sin \alpha \pi}{\pi} 
        \int_0^1 \frac{t^{\alpha-1}h_{-\alpha nx}^{(m)}(-t)}{|(t(1-x))^\alpha - x^\alpha e^{-i \alpha \pi}|^2} \, \dd t \nn \\
        &\quad+(-1)^m \frac{\alpha x^{\alpha(n+1)}(1-x)^\alpha \sin \alpha \pi}{\pi}
        \int_0^1 \frac{t^{\alpha-1}h_{-\alpha n(1-x)}^{(m)}(-t)}{|(1-x)^{\alpha}e^{-i \alpha \pi}-(tx)^\alpha|^2} \, \dd t\Bigg\} \nn \\
        &=: \I_{\alpha, m}^{(n)}(x) 
        + \II_{\alpha, m}^{(n)}(x)
        -\III_{\alpha, m}^{(n)}(x). 
        \label{Eq:Estimate-I+II+III}
    \end{align}
    We note here that 
    \begin{align*}
        M_{\alpha, m}^{(n)}(1-x)
        &= (-1)^m  \alpha \sum_{j=0}^n 
        \binom{\alpha n}{\alpha j}(1-x)^{\alpha j}x^{\alpha(n-j)}\left(\frac{n-j}{n}-x\right)^m \\
        &=(-1)^m  \alpha \sum_{k=0}^n 
        \binom{\alpha n}{\alpha n -\alpha k}(1-x)^{\alpha (n-k)}x^{\alpha k}
        \left(\frac{k}{n}-x\right)^m \\
        &=(-1)^m M_{\alpha, m}^{(n)}(x)
    \end{align*}
    and 
    \begin{align*}
        \III_{\alpha, m}^{(n)}(1-x)
        &= \frac{\alpha (1-x)^\alpha
        x^{\alpha(n+1)}\sin \alpha \pi}{\pi} 
        \int_0^1 \frac{t^{\alpha-1}h_{-\alpha n(1-x)}^{(m)}(-t)}{|(tx)^\alpha - (1-x)^\alpha e^{-i \alpha \pi}|^2} \, \dd t \\
        &\quad+(-1)^m \frac{\alpha (1-x)^{\alpha(n+1)}x^\alpha \sin \alpha \pi}{\pi}
        \int_0^1 \frac{t^{\alpha-1}h_{-\alpha nx}^{(m)}(-t)}{|x^{\alpha}e^{-i \alpha \pi}-(t(1-x))^\alpha|^2} \, \dd t \\
        &=(-1)^m \III_{\alpha, m}^{(n)}(x)
    \end{align*}
    for $x \in (0, 1)$. 
    Furthermore, $\II_{\alpha, m}(x)=0$ if $\alpha \in (0, 2)$ because $K_\alpha=\{1\}$. 
    Therefore, it follows from \eqref{Eq:Estimate-I+II+III} that for $\alpha \in (0, 2)$, 
    \begin{equation}\label{Eq:symmetry-1}
        \I_{\alpha, m}^{(n)}(1-x)=
        (-1)^m\I_{\alpha, m}^{(n)}(x). 
    \end{equation}
    Because $\I_{\alpha, m}^{(n)}(x)$
    is a polynomial in $\alpha$ thanks to Lemma~\ref{Lem:Representation-h_c^m}, 
    \eqref{Eq:symmetry-1} holds true for arbitrary $\alpha>0$, 
    which in turn implies that 
    \begin{equation}\label{Eq:symmetry-2}
        \II_{\alpha, m}^{(n)}(1-x)
        =(-1)^m\II_{\alpha, m}^{(n)}(x), \qquad \alpha>0. 
    \end{equation}
    In what follows, the integer $n \in \N$ is taken to be sufficiently large 
    so that $\alpha n \ge m+1$. 

    \smallskip
    \noindent
    {\bf Step 2.}
    We give a uniform estimate of $\II_{\alpha, m}^{(n)}(x)$. 
    Suppose that $\omega \in K_\alpha \setminus \{1\}$ and $x \in (0, 1/2]$. 
    For fixed $k \in \N_0$ with $0 \le k \le m$, we have 
    \begin{align*}
        &|W_{-\alpha nx}(m, k)(\alpha n)_k(x\omega)^k(1-x+x\omega)^{\alpha n-k}| \\
        &\le C_{2,1}\{(\alpha nx)^{m-k}+1\} 
        (\alpha n)^k x^k \{1-2x(1-x)(1-\mathrm{Re}(\omega))\}^{(\alpha n-k)/2} \\
        &\le C_{2,1}\{(\alpha nx)^m + (\alpha nx)^k\}\{1-x(1-\mathrm{Re}(\omega))\}^{(\alpha n-k)/2} \\
        &\le C_{2,1}\left(\frac{m}{e (\alpha n-k)/2}\right)^m
        \left(\frac{\alpha n}{1-\mathrm{Re}(\omega)}\right)^m
        +C_{2,1}\left(\frac{k}{e (\alpha n-k)/2}\right)^k
        \left(\frac{\alpha n}{1-\mathrm{Re}(\omega)}\right)^k \\
        &\le C_{2,5}
    \end{align*}
    by using \eqref{Eq:Estimate-1-x+xomega}, \eqref{Eq:Estimate-W_c}, and Lemma~\ref{Lem:pre-ineq}.
    Thus, we have
    \[
        \sup_{x \in (0, 1/2]}|\II_{\alpha, m}^{(n)}(x)|
        \le (m+1)(\#(K_\alpha)-1)C_{2,5} \le C_{2,6}, 
    \]
    where $C_{2,6}>0$ depends on only $m \in \N$ and a compact set $K$ in which $\alpha$ takes values. 
    By \eqref{Eq:symmetry-2}, we obtain that
    \begin{equation}\label{Eq:estimate of II}
        \sup_{x \in (0, 1)}|\II_{\alpha, m}^{(n)}(x)|=
        \sup_{x \in (0, 1/2]}|\II_{\alpha, m}^{(n)}(x)| \le C_{2,6}.
    \end{equation}
     
    \smallskip
    \noindent
    {\bf Step 3.}
    We next give a uniform estimate of $\III_{\alpha, m}^{(n)}(x)$.
    We write 
    \[
        \III_{\alpha, m}^{(n)}(x)
        =\widetilde{\III}_{\alpha, m}^{(n)}(x)
        +(-1)^m \widetilde{\III}_{\alpha, m}^{(n)}(1-x) 
    \]
    with 
    \[
        \widetilde{\III}_{\alpha, m}^{(n)}(x)
        =\frac{\alpha x^\alpha
        (1-x)^{\alpha(n+1)}\sin \alpha \pi}{\pi} 
        \int_0^1 \frac{t^{\alpha-1}h_{-\alpha nx}^{(m)}(-t)}{|(t(1-x))^\alpha - x^\alpha 
        e^{-i \alpha \pi}|^2} \, \dd t.
    \]
    It follows from \eqref{Eq:Estimate-W_c} that 
    \begin{align*}
        |h_{-\alpha nx}^{(m)}(-t)|
        &= \left|\sum_{k=0}^m W_{-\alpha nx}(m, k)(-\alpha nx)(\alpha n)_k
        (-t)^k (1-t)^{\alpha n-k}\right| \\
        &\le C_{2,7}\sum_{k=0}^m \{(\alpha nx)^{m-k}+1\}(\alpha n)^k t^k(1-t)^{\alpha n-k}.
    \end{align*}
    Thus, we have 
    \begin{align*}
        |\widetilde{\III}_{\alpha, m}^{(n)}(x)|
        &\le C_{2,8}\sum_{k=0}^m x^\alpha (1-x)^{\alpha (n+1)} |\sin \alpha \pi| \\
        &\quad\times \{(nx)^{m-k}+1\}
        \int_0^1 \frac{t^{\alpha-1} n^kt^{k}(1-t)^{\alpha n-k}}{|(t(1-x))^\alpha - x^\alpha e^{-i \alpha \pi}|^2} \, \dd t.
    \end{align*}
    By noting 
    \[
        n^kt^{k}(1-t)^{\alpha n-k}
        \le n^k \left(\frac{k}{e (\alpha n-k)}\right)^k \le C_{2,9}
    \]
    and Lemma~\ref{Lem:Estimates-integral-terms}, we have 
    \begin{align*}
        |\widetilde{\III}_{\alpha, m}^{(n)}(x)|
        &\le C_{2,10}\sum_{k=0}^m (1-x)^{\alpha n} 
        \{(nx)^{m-k}+1\} \\
        &= C_{2,10}\sum_{k=0}^m \{(1-x)^{\alpha n}(nx)^{m-k}+(1-x)^{\alpha n}\} \\
        &\le C_{2,10}\sum_{k=0}^m \left\{ \left(\frac{m-k}{e\alpha n}\right)^{m-k}n^{m-k}+1\right\} \le C_{2,11},
    \end{align*}
    from which we conclude that 
    \begin{equation}\label{Eq:estimate of III}
        \sup_{x \in (0, 1)}|\III_{\alpha, m}^{(n)}(x)| \le C_{2,12}, 
    \end{equation}
    where $C_{2,12}>0$ depends on only $m \in \N$ and $K$ in which $\alpha$ takes values.  
    
    By combining \eqref{Eq:Estimate-I+II+III} with 
    \eqref{Eq:estimate of II} and \eqref{Eq:estimate of III}, we obtain that
    \[
        \sup_{x \in (0, 1)}|(\alpha n)^m M_{\alpha, m}^{(n)}(x) 
        - \I_{\alpha, m}^{(n)}(x)|
        = \sup_{x \in (0, 1)}|\II_{\alpha, m}^{(n)}(x)
        -\III_{\alpha, m}^{(n)}(x)| \le C_{2,13}.
    \]
    This implies
    \eqref{Eq:Uniform-moment-estimate} once we divide both sides of the above inequality by $(\alpha n)^m$ and take the supremum over $\alpha \in K$. 
\end{proof}

We also have the following 
by applying Theorems~\ref{Thm:Uniform-bound-normalized-constant} and \ref{Thm:Unifrom-moment-estimate}, and \eqref{Eq:Examples-leading-terms}.

\begin{co}\label{Cor:Moment-estimate}
    Let $K$ be a compact subset of $(0, \infty)$. As $n \to \infty$, 
    \[
        \sup_{\alpha \in K} \max_{x \in [0, 1]}
        \left|\E\left[ \left(\frac{S_{\alpha, x}^{(n)}}{n}-x \right)^m\right]\right|
        =\begin{cases}
        \displaystyle O\left(\frac{1}{n}\right) & \text{if }m=1, 2, \\
        \displaystyle O\left(\frac{1}{n^2}\right) & \text{if }m=3, 4. 
        \end{cases}
    \]
\end{co}

\subsection{Characteristic function of the fractional binomial distribution}
\label{Subsect:CF}

In this section, we obtain a quantitative estimate of the characteristic function of 
the fractional binomial distribution $\mu_{\alpha, x}^{(n)}$, that is, 
    \[
        \varphi_{\alpha, x}^{(n)}(\xi)
        :=\sum_{j=0}^n e^{i\xi j}\mu_{\alpha, x}^{(n)}(\{j\}), \qquad \xi \in \R.
    \]
We define $\underline{\alpha}:=2\lfloor \alpha/2 \rfloor$
and $\ol{\alpha}:=2\lceil \alpha/2 \rceil$ for $\alpha>0$.
Moreover, we put 
    \[
        \theta_\alpha:=
        \begin{cases}
        \{(\alpha-\underline{\alpha}) \wedge (\ol{\alpha}-\alpha)\}\pi 
        & \text{if $\alpha \notin \N$}, \\ 
        2\pi &  \text{if $\alpha \in \N$}.
        \end{cases}
    \]
We then have the following lemma. 

\begin{lm}\label{Lem:analytic continuation}
    Let $\alpha>0$. 

    \smallskip
    \noindent
    {\rm (1)} If $\alpha \in \N$, we have 
    \begin{equation}\label{Eq:analytic-continuation-2}
        \alpha\sum_{j=0}^n \binom{\alpha n}{\alpha j}\lambda^{\alpha j} 
        =\sum_{\omega \in K_\alpha}(1+\lambda \omega)^{\alpha n}
    \end{equation}
    for $\lambda \in \mathbb{C}$. 
    
    \smallskip
    \noindent
    {\rm (2)} If $\alpha \notin \N$, we have 
    \begin{align}
        &\alpha\sum_{j=0}^n \binom{\alpha n}{\alpha j}\lambda^{\alpha j} \nn \\
        &=\sum_{\omega \in K_\alpha}(1+\lambda \omega)^{\alpha n}
        -\frac{\alpha\lambda^\alpha \sin \alpha \pi}{\pi}
        \int_0^1 t^{\alpha-1}(1-t)^{\alpha n} \nn \\
        &\quad\times\Bigg\{
        \frac{1}{(t^\alpha-\lambda^\alpha e^{-i\alpha \pi})(t^\alpha-\lambda^\alpha e^{i\alpha \pi})}
        +\frac{\lambda^{\alpha (n-2)}}{(\lambda^{-\alpha}e^{-i\alpha \pi}- t^\alpha)(\lambda^{-\alpha}e^{i\alpha \pi}- t^\alpha)}\Bigg\} \, \dd t
        \label{Eq:analytic-continuation-1}
    \end{align}
    for $\lambda \in D_\alpha:=\{re^{i\theta} \mid r>0, \, |\theta| < \theta_\alpha/\alpha\}$.
\end{lm}

\begin{proof}
    We assume $\alpha \in \N$. Then, \eqref{Eq:analytic-continuation-2} holds 
    for all $\lambda \in (0, 1]$ by Proposition \ref{Prop:generalization-binomial-theorem-HH}. 
    Because both sides of \eqref{Eq:analytic-continuation-2} are holomorphic on $\mathbb{C}$, 
    \eqref{Eq:analytic-continuation-2} holds for all
    $\lambda \in \mathbb{C}$ by the identity theorem for holomorphic functions.
    
    We next assume $\alpha \notin \N$.  
    It then follows from Proposition~\ref{Prop:generalization-binomial-theorem-HH} 
    that \eqref{Eq:analytic-continuation-1} holds for $\lambda \in (0, 1]$. 
    We now find a domain of $\lambda \in \mathbb{C}$ including $(0, 1)$ in which 
    the functions defined by both sides of \eqref{Eq:analytic-continuation-1} are holomorphic. 
    Let us put $\lambda=re^{i\theta}$ with $r > 0$ and $|\theta|<\pi$. 
    Then, it suffices to confirm that 
    $\lambda \in \mathbb{C} \setminus (-\infty, 0]$,  
    $\lambda \omega \in \mathbb{C} \setminus (-\infty, 0]$ for all $\omega \in K_\alpha$, 
    $\lambda^\alpha e^{\pm i\alpha \pi} \notin [0, \infty)$, and $\lambda^{-\alpha} e^{\pm i\alpha \pi} \notin [0, \infty)$.
    The condition $\lambda \in \mathbb{C} \setminus (-\infty, 0]$ 
    merely says that $|\theta|<\pi$. 
    By noting that 
        \[
            K_\alpha=\left\{\exp\left(i\frac{2k\pi}{\alpha}\right) \, : \, 
            k \in \Z, \, -\left\lfloor \frac{\alpha}{2} \right\rfloor
            \le k \le \left\lfloor \frac{\alpha}{2} \right\rfloor \right\}, 
        \]
    the condition $\lambda \omega \in \mathbb{C} \setminus (-\infty, 0]$ 
    for all $\omega \in K_\alpha$ is satisfied if
        \[
            \theta \in (-\pi, \pi) \cap  \bigcap_{k \in \Z, -\lfloor \alpha /2\rfloor \le k \le \lfloor \alpha /2\rfloor} \left(-\pi-\frac{2k\pi}{\alpha}, 
            \pi-\frac{2k\pi}{\alpha}\right) =\left(-\frac{\alpha-\underline{\alpha}}{\alpha}\pi, \frac{\alpha-\underline{\alpha}}{\alpha}\pi\right).
        \]
    Moreover, the conditions 
    $\lambda^\alpha e^{\pm i\alpha \pi}\notin [0, \infty)$ and $\lambda^{-\alpha} e^{\pm i\alpha \pi} \notin [0, \infty)$ are satisfied if
        \begin{equation}\label{Eq:argument}
            \arg(e^{\mp i\alpha \pi})-2\pi < \alpha \theta < \arg(e^{\mp i\alpha \pi}) 
            \quad \text{and} \quad 
            \arg(e^{\mp i\alpha \pi})-2\pi < -\alpha \theta < \arg(e^{\mp i\alpha \pi}),
        \end{equation}
    where $\arg(e^{\mp i\alpha \pi})$ is taken from $(0, 2\pi)$.
    Because $\arg(e^{-i\alpha \pi})=(\ol{\alpha}-\alpha)\pi$ and 
    $\arg(e^{i\alpha \pi})=(\alpha-\underline{\alpha})\pi$, \eqref{Eq:argument} is rewritten as $-\eta_\alpha < \theta < \eta_\alpha$, where
        \[
            \eta_\alpha
            =\big[ 
            \{2-(\alpha-\underline{\alpha})\} \wedge \{2-(\ol{\alpha}-\alpha)\} \wedge (\ol{\alpha}-\alpha) \wedge (\alpha-\underline{\alpha})
            \big]\frac{\pi}{\alpha}.
        \]
    By noting that $\eta_\alpha=\theta_\alpha/\alpha$ 
    because $\ol{\alpha}-\underline{\alpha}=2$, 
    both sides of \eqref{Eq:analytic-continuation-1} turn out to be holomorphic 
    in the domain $D_\alpha$. 
    Thus, \eqref{Eq:analytic-continuation-1} also holds for all
    $\lambda \in D_\alpha$ by the identity theorem.
\end{proof}

We provide a simple lemma for the proof of Theorem~\ref{Thm:CF-fractional-binomial} below.

\begin{lm}\label{Lem:simple-ineq}
For $a \ge 0$, $b \ge 0$ and $\theta \in \R$, we have
    \[
        |a-be^{i\theta}| \ge \sqrt{a^2+b^2}\left|\sin\frac{\theta}{2}\right|. 
    \]
\end{lm}

\begin{proof}\pushQED{\qed}
    Because $|a-be^{i\theta}|=|b-ae^{-i\theta}|$, we may assume $a \le b$. 
    First, suppose $\cos \theta \le 0$. Then, we have 
    \[
        |a-be^{i\theta}|^2=a^2+b^2-2ab\cos \theta \ge a^2+b^2 \ge (a^2+b^2)\sin^2\frac{\theta}{2}.
    \]
    Next, suppose $\cos \theta>0$. Then, it holds that $\cos^2(\theta/2)=(1+\cos \theta)/2>1/2$.
    Therefore, we have 
    \begin{align}
        |a-be^{i\theta}|^2&=(a-b\cos \theta)^2+b^2\sin^2\theta \nn \\ 
        &=(a-b\cos \theta)^2+4b^2\sin^2\frac{\theta}{2}\cos^2\frac{\theta}{2} \nn \\
        &\ge 2b^2\sin^2\frac{\theta}{2} \ge (a^2+b^2)\sin^2\frac{\theta}{2}.\tag*{\qedhere}
    \end{align}
    \end{proof}

A quantitative estimate of $\varphi_{\alpha, x}^{(n)}$ is now stated as follows.

\begin{tm}\label{Thm:CF-fractional-binomial}
    Let $\alpha>0$, $n \in \N$, $x \in (0, 1)$,
    and $\xi \in \R$ satisfy that $|\xi|<\theta_\alpha$.
    Then, we have 
    \begin{align}
        \varphi_{\alpha, x}^{(n)}(\xi)
        &=\frac{1}{Z_{\alpha, x}^{(n)}}\left(1-x+xe^{i\xi/\alpha}\right)^{\alpha n} 
        +\frac{1}{Z_{\alpha, x}^{(n)}}
        \sum_{\omega \in K_\alpha \setminus \{1\}}
        \left(1-x+x\omega e^{i\xi/\alpha}\right)^{\alpha n} \nn \\
        &\quad-\frac{1}{Z_{\alpha, x}^{(n)}}
        \frac{\alpha e^{i\xi}\sin \alpha \pi}{\pi}\int_0^1 t^{\alpha-1}(1-t)^{\alpha n} \nonumber \\
        &\quad\times \Bigg[
        \frac{x^\alpha (1-x)^{\alpha(n+1)}}{\{(t(1-x))^\alpha-x^\alpha e^{i(\xi-\alpha \pi)}\}\{(t(1-x))^\alpha-x^\alpha e^{i(\xi+\alpha \pi)}\}} \nn \\
        &\qquad
        +\frac{x^{\alpha (n+1)}(1-x)^{\alpha}e^{i\xi (n-2)}}
        {\{(1-x)^\alpha e^{-i(\xi+\alpha \pi)}-(tx)^\alpha\}\{(1-x)^\alpha e^{-i(\xi-\alpha\pi)}-(tx)^\alpha\}}
        \Bigg] \, \dd t 
        \label{Eq:CF-fractional-binomial-1}\\
        &=\frac{1}{Z_{\alpha, x}^{(n)}}
        \left(1-x+xe^{i\xi/\alpha}\right)^{\alpha n} + O(e^{-\delta n})
        \label{Eq:CF-fractional-binomial-2}
    \end{align}
    as $n \to \infty$ for some $\delta>0$ that is chosen uniformly 
    on any compact subset of $x \in (0, 1)$
    and any compact subset of $\xi \in (-\theta_\alpha, \theta_\alpha)$. 
    Here, if $\alpha \in \N$, then the last term of \eqref{Eq:CF-fractional-binomial-1} 
    is regarded as $0$. 
\end{tm}

\begin{proof}
We have 
    \begin{align*}
        \varphi_{\alpha, x}^{(n)}(\xi)
        &=\sum_{j=0}^n e^{i\xi j} \times 
        \frac{\alpha}{Z_{\alpha, x}^{(n)}}
        \binom{\alpha n}{\alpha j}x^{\alpha j}(1-x)^{\alpha(n-j)} \\
        &=\frac{\alpha}{Z_{\alpha, x}^{(n)}}
        \sum_{j=0}^n \binom{\alpha n}{\alpha j}
        \left(\frac{xe^{i\xi/\alpha}}{1-x}\right)^{\alpha j}(1-x)^{\alpha n}. 
    \end{align*}
We put $\lambda=xe^{i\xi/\alpha}/(1-x)$. 
We first consider the case where $\alpha \notin \N$.
From Lemma~\ref{Lem:analytic continuation}, 
    \begin{align*}
        \varphi_{\alpha, x}^{(n)}(\xi) 
        &=\frac{1}{Z_{\alpha, x}^{(n)}}(1-x)^{\alpha n}  
        \alpha \sum_{j=0}^n \binom{\alpha n}{\alpha j} \lambda^{\alpha j} \\
        &=\frac{1}{Z_{\alpha, x}^{(n)}}(1-x)^{\alpha n} \Bigg[ 
        \sum_{\omega \in K_\alpha}(1+\lambda\omega)^{\alpha n}
        -\frac{\alpha \lambda^\alpha \sin \alpha \pi}{\pi}
        \int_0^1 t^{\alpha-1}(1-t)^{\alpha n} \nonumber \\
        &\quad\times \left\{
        \frac{1}{(t^\alpha-\lambda^\alpha e^{-i\alpha \pi})(t^\alpha-\lambda^\alpha e^{i\alpha \pi})}
        +\frac{\lambda^{\alpha (n-2)}}{(\lambda^{-\alpha}e^{-i\alpha \pi}-t^\alpha)(\lambda^{-\alpha}e^{i\alpha \pi}-t^\alpha)}
        \right\} \, \dd t \Bigg] \\
        &=\frac{1}{Z_{\alpha, x}^{(n)}}\left(1-x+xe^{i\xi/\alpha}\right)^{\alpha n}
        +\frac{1}{Z_{\alpha, x}^{(n)}}
        \sum_{\omega \in K_\alpha \setminus \{1\}}
        \left(1-x+x\omega e^{i\xi/\alpha}\right)^{\alpha n} \\
        &\quad-\frac{1}{Z_{\alpha, x}^{(n)}}
        \frac{\alpha x^\alpha (1-x)^{\alpha(n-1)}e^{i\xi}\sin \alpha \pi}{\pi}\int_0^1 t^{\alpha-1}(1-t)^{\alpha n} \nonumber \\
        &\quad\times \Bigg[
        \frac{(1-x)^{2\alpha}}{\{(t(1-x))^\alpha-x^\alpha e^{i(\xi-\alpha \pi)}\}\{(t(1-x))^\alpha-x^\alpha e^{i(\xi+\alpha \pi)}\}} \\
        &\qquad+\frac{x^{\alpha n}(1-x)^{-\alpha(n-2)}e^{i\xi n}}
        {\{(1-x)^\alpha e^{-i\alpha \pi}-(tx)^\alpha e^{i\xi}\}\{(1-x)^\alpha e^{i\alpha \pi}-(tx)^\alpha e^{i\xi}\}}
        \Bigg] \, \dd t,
    \end{align*}
which implies \eqref{Eq:CF-fractional-binomial-1}. 
Next, we prove \eqref{Eq:CF-fractional-binomial-2}. We have
    \begin{align*}
        \left|1-x+x\omega e^{i\xi/\alpha}\right|^2
        &=(1-x)^2+2x(1-x)\mathrm{Re}\left(\omega e^{i\xi/\alpha}\right)+x^2 \\
        &=1-2x(1-x)\left\{ 1-\mathrm{Re}\left(\omega e^{i\xi/\alpha}\right)\right\} <1
    \end{align*}
for $\omega \in K_\alpha \setminus \{1\}$ because $\mathrm{Re}(\omega e^{i\xi/\alpha})<1$
from the assumption of $\xi$. 
Here, we also note that $K_\alpha \setminus \{1\} \neq \varnothing$ only if $\alpha \ge 2$. 
Moreover, from Lemma \ref{Lem:simple-ineq}, we have 
    \[
        |(t(1-x))^\alpha - x^\alpha e^{i(\xi \pm \alpha \pi)}| \ge x^\alpha 
        \left|\sin \frac{\xi \pm \alpha \pi}{2}\right|
    \]
and 
    \[
        |(1-x)^\alpha e^{-i(\xi \pm \alpha \pi)} - (tx)^\alpha| 
        \ge (1-x)^\alpha \left|\sin \frac{\xi \pm \alpha \pi}{2}\right|.
    \]
Note that $|\xi|<\theta_\alpha$ implies $\sin (\xi \pm \alpha \pi)/2 \neq 0$. 
From these estimates and Proposition~\ref{Prop:Order-moment-fractional-binomial}~(1), we obtain \eqref{Eq:CF-fractional-binomial-2}. 
The case for $\alpha \in \N$ is simpler to prove. 
\end{proof}

\subsection{The fractional multinomial distribution}
\label{Subsect:multinomial}

Let $d \ge 2$ be an integer. 
For $\bx=(x_1, x_2, \dots, x_d) \in \R^d$, we write 
    \[
        |\bx|:=x_1+x_2+\cdots+x_d, \qquad 
        \|\bx\| := (x_1^2+x_2^2+\cdots+x_d^2)^{1/2}.
    \]
We define the $(d-1)$-simplex in $\R^d$ by
    \[
        \Delta_{d-1}:=\left\{
        \bx = (x_1, x_2, \dots, x_d) \in \R^d \,
        | \,  x_i \ge 0, i=1, 2, \dots, d, \, |\bx| = 1 \right\}.
    \]
For $n \in \N$, we also set 
    \[
        J_d^{(n)}=\{\bm{j}=(j_1, j_2, \dots, j_d) \mid j_k \in \N_0, 
        \, k=1, 2, \dots, d, \, |\bm{j}|=n\}. 
    \]

According to the preliminary version \cite[Lemma~1]{FR} of the paper \cite{FR-final}, 
we can generalize the neo-classical inequality presented in \cite[Theorem~1.2]{HH}
to the multivariate cases, which is stated as follows.

\begin{tm}[Multivariate neo-classical inequality; see also {\cite[Exercise 3.9]{LCL}}]
\label{Thm:multivariate neo-classical inequality}
    Let $\alpha \in (0, 2)$, $n \in  \N$ and $x_1, x_2, \dots, x_d \ge 0$. 
    Then, we have 
    \begin{align}
        &\alpha^d \sum_{\bm{j} \in J_d^{(n)}}
        \frac{\Gamma(\alpha n+1)}{\Gamma(\alpha j_1+1)\Gamma(\alpha j_2+1)\cdots\Gamma(\alpha j_d+1)}
        x_1^{\alpha j_1}x_2^{\alpha j_2} \cdots x_d^{\alpha j_d} \nn \\
        &\begin{cases}
        \le (x_1+x_2+\cdots+x_d)^{\alpha n} &\text{if }\alpha \in (0, 1], \\
        \ge (x_1+x_2+\cdots+x_d)^{\alpha n} &\text{if }\alpha \in (1, 2). \\
        \end{cases}
        \label{Eq:multivariate neo-classical inequality}
    \end{align}
    The equality holds if and only if $\alpha=1$ or $x_1=x_2=\cdots=x_d=0$. 
\end{tm}

\begin{proof}
    The proof is by induction on $d \ge 2$. 
    For $d=2$, \eqref{Eq:multivariate neo-classical inequality} is nothing but the neo-classical inequality and its converse. 
    Next, we assume that \eqref{Eq:multivariate neo-classical inequality} holds for some $d >2$. 
    Then for $\alpha \in (0, 1]$, we have 
    \begin{align*}
        &\alpha^{d+1} \sum_{\bm{j} \in J_{d+1}^{(n)}}
        \frac{\Gamma(\alpha n+1)}{\Gamma(\alpha j_1+1)\Gamma(\alpha j_2+1)\cdots\Gamma(\alpha j_{d+1}+1)}
        x_1^{\alpha j_1}x_2^{\alpha j_2} \cdots x_{d+1}^{\alpha j_{d+1}} \nn \\
        &=\alpha
        \sum_{j_{d+1}=0}^n \frac{\Gamma(\alpha n+1)}{\Gamma(\alpha(n-j_{d+1})+1)\Gamma(\alpha j_{d+1}+1)} x_{d+1}^{\alpha j_{d+1}} \nn \\ &\quad\times \alpha^d
        \sum_{(j_1, j_2, \dots, j_d) \in J_d^{(n-j_{d+1})}}
        \frac{\Gamma(\alpha (n-j_{d+1})+1)}{\Gamma(\alpha j_1+1)\Gamma(\alpha j_2+1)\cdots\Gamma(\alpha j_{d}+1)}
        x_1^{\alpha j_1}x_2^{\alpha j_2} \cdots x_{d}^{\alpha j_{d}} \nn \\
        &\le \alpha
        \sum_{j_{d+1}=0}^n \binom{\alpha n}{\alpha j_{d+1}} (x_1+x_2+\cdots+x_d)^{\alpha(n-j_{d+1})}x_{d+1}^{\alpha j_{d+1}} \nn \\
        &\le (x_1+x_2+\cdots+x_{d+1})^{\alpha n}.
    \end{align*}
    The proof for $\alpha \in (1, 2)$ is similar to that 
    for $\alpha \in (0, 1]$. 
\end{proof}

Needless to say, if $\alpha=1$, the \eqref{Eq:multivariate neo-classical inequality}
reads as usual multinomial theorem. 
As with Definition~\ref{Def:fractional-binomial-distribution}, 
this multivariate neo-classical inequality allows us to define 
the fractional analogue of the multinomial distribution. 

\begin{df}[Fractional multinomial distribution]
\label{Def:fractional multinomial distribution}
    Let $\alpha > 0$, $n \in \N$ and $\bx=(x_1, x_2, \dots, x_d) \in \Delta_{d-1}$. 
    Then, a probability measure $\mu_{\alpha, \bx}^{(n)}$ 
    is called the $\alpha$-fractional multinomial distribution if 
    \begin{equation}\label{Eq:fractional multinomial distribution}
        \mu_{\alpha, \bx}^{(n)}(\dd \bm{z})
        :=\frac{\alpha^d}{Z_{\alpha, \bx}^{(n)}}
        \sum_{\bm{j} \in J_d^{(n)}}
        \frac{\Gamma(\alpha n+1)}{\Gamma(\alpha j_1+1)\Gamma(\alpha j_2+1)\cdots\Gamma(\alpha j_d+1)}
        x_1^{\alpha j_1}x_2^{\alpha j_2} \cdots x_d^{\alpha j_d}
        \delta_{\bm{j}}(\dd \bm{z}),
    \end{equation}
    where $Z_{\alpha, \bm{x}}^{(n)}$ is the normalized constant given by 
    \[
        Z_{\alpha, \bx}^{(n)}=\alpha^d \sum_{\bm{j} \in J_d^{(n)}}
        \frac{\Gamma(\alpha n+1)}{\Gamma(\alpha j_1+1)\Gamma(\alpha j_2+1)\cdots\Gamma(\alpha j_d+1)}
        x_1^{\alpha j_1}x_2^{\alpha j_2} \cdots x_d^{\alpha j_d} \, (>0). 
    \]
\end{df}

\noindent
If $\alpha=1$, then \eqref{Eq:fractional multinomial distribution} becomes the 
usual definition of the multinomial distribution. 
However, we have no explicit formula 
for the remainder term $R^{(n)}_\alpha(x_1, x_2, \dots, x_d)$ satisfying 
    \begin{align*}
        &\alpha^d \sum_{\bm{j} \in J_d^{(n)}}
        \frac{\Gamma(\alpha n+1)}{\Gamma(\alpha j_1+1)\Gamma(\alpha j_2+1)\cdots\Gamma(\alpha j_d+1)}
        x_1^{\alpha j_1}x_2^{\alpha j_2} \cdots x_d^{\alpha j_d} \\
        &=(x_1+x_2+\cdots+x_d)^{\alpha n}
        +R^{(n)}_\alpha(x_1, x_2, \dots, x_d)
    \end{align*}
for $d>2$. 
Obtaining such a formula explicitly might require a multivariate version of the fractional Taylor series, which, although it appears rather difficult, may nevertheless be worth further investigation.

\section{{\bf Limit theorems for the fractional binomial distribution}}
\label{Sect:limit theorems}

\subsection{The weak law of large numbers}
\label{Subsect:WLLN}

First, we prove the weak law of large numbers 
for the fractional binomial distribution $\mu_{\alpha, x}^{(n)}$. 

\begin{tm}[Weak law of large numbers for $\mu_{\alpha, x}^{(n)}$]
\label{Thm:WLLN}
Let $\alpha>0$, $n \in \N$, and $x \in [0, 1]$. 
Let $S_{\alpha, x}^{(n)}$ denote the random variable whose law is $\mu_{\alpha, x}^{(n)}$. 
Then, for every $\ve>0$ and a compact subset $K$ of $(0, \infty)$, we have  
    \begin{equation}\label{Eq:WLLN}
        \lim_{n \to \infty}
        \sup_{\alpha \in K} \sup_{x \in [0, 1]} \mathbb{P}
        \left(\left|\frac{1}{n}S_{\alpha, x}^{(n)}-x\right| \ge \ve\right)=0. 
    \end{equation}
\end{tm}

\begin{proof}
    It follows from the Chebyshev inequality, Theorem~\ref{Thm:Uniform-bound-normalized-constant}, and Corollary~\ref{Cor:Moment-estimate} that 
    \begin{align*}
        &\mathbb{P}
        \left(\left|\frac{1}{n}S_{\alpha, x}^{(n)}-x\right| \ge \ve\right) \\
        &\le \frac{1}{\ve^2}\mathbb{E}\left[ \left(\frac{1}{n}S_{\alpha, x}^{(n)}-x\right)^2\right] \\
        &= \frac{1}{\ve^2}
        \sum_{j=0}^n \frac{\alpha}{Z_{\alpha, x}^{(n)}}
        \binom{\alpha n}{\alpha j}x^{\alpha j}(1-x)^{\alpha(n-j)}
        \left(\frac{j}{n}-x\right)^2 \\
        &\le \frac{1}{\ve^2} \left(\sup_{n \in \N}\sup_{\alpha \in K}\max_{x \in [0, 1]}\frac{1}{Z_{\alpha, x}^{(n)}}\right) \sup_{\alpha \in K}\max_{x \in [0, 1]}
        \left|\alpha \sum_{j=0}^n \binom{\alpha n}{\alpha j}x^{\alpha j}(1-x)^{\alpha(n-j)}
        \left(\frac{j}{n}-x\right)^2\right| \nn \\
        &=O\left(\frac{1}{n}\right)
    \end{align*}
    as $n \to \infty$, which implies \eqref{Eq:WLLN}.
\end{proof}

\subsection{The central limit theorem}
\label{Subsect:CLT}

We prove the central limit theorem for the fractional binomial distribution 
$\mu_{\alpha, x}^{(n)}$, which is stated as follows.

\begin{tm}[Central limit theorem for $\mu_{\alpha, x}^{(n)}$]
\label{Thm:CLT}
    Let $\alpha>0$, $n \in \N$, and $x \in (0, 1)$.
    Let $\widetilde{\mu}_{\alpha, x}^{(n)}$ denote the law of the random variable 
        \[
            \widetilde{S}_{\alpha, x}^{(n)}
            :=\frac{S_{\alpha, x}^{(n)} - \mathbb{E}[S_{\alpha, x}^{(n)}]}{\sqrt{\mathrm{Var}(S_{\alpha, x}^{(n)})}}.
        \] 
    Then, $\widetilde{\mu}_{\alpha, x}^{(n)}$
    converges weakly to the standard normal distribution $N(0, 1)$
    as $n \to \infty$. 
\end{tm}

\begin{proof}
    We write $m_n=\mathbb{E}[S_{\alpha, x}^{(n)}]$ and $v_n=\mathrm{Var}(S_{\alpha, x}^{(n)})$. 
    Let $\widetilde{\varphi}_{\alpha, x}^{(n)}$ denote the characteristic function of $\widetilde{\mu}_{\alpha, x}^{(n)}$, and 
    let $\xi \in \R$. 
    We note that $v_n$ diverges as $n \to \infty$ from Corollary~\ref{Cor:Mean-variance-fractional-binomial}. 
    Hence, it holds that $|\xi/\sqrt{v_n}| < \theta_\alpha$ for sufficiently large $n$, and
    we consider only such $n$ in the following. 
    By applying Theorem~\ref{Thm:CF-fractional-binomial}, we have 
    \begin{align}
        &\widetilde{\varphi}_{\alpha, x}^{(n)}(\xi)  \nn \\
        &=\exp\left(-\frac{i\xi m_n}{\sqrt{v_n}}\right)
        \varphi_{\alpha, x}^{(n)}\left(\frac{\xi}{\sqrt{v_n}}\right) \nn \\
        &=\exp\left(-\frac{i\xi m_n}{\sqrt{v_n}}\right)
        \Bigg[\frac{1}{Z_{\alpha, x}^{(n)}}
        \left(1-x+x\exp\left(\frac{i\xi}{\alpha\sqrt{v_n}}\right)\right)^{\alpha n} \nn \\
        &\quad+\frac{1}{Z_{\alpha, x}^{(n)}}
        \sum_{\omega \in K_\alpha \setminus \{1\}}
        \left(1-x+x\omega\exp\left(\frac{i\xi}{\alpha\sqrt{v_n}}\right)\right)^{\alpha n} \nn \\
        &\quad
        -\frac{1}{Z_{\alpha, x}^{(n)}}
        \frac{\alpha \sin \alpha \pi}{\pi}
        \exp\left(\frac{i\xi}{\sqrt{v_n}}\right)
        \int_0^1 t^{\alpha-1}(1-t)^{\alpha n} \nonumber \\
        &\quad\times \Bigg\{\frac{x^\alpha (1-x)^{\alpha(n+1)}}{\{(t(1-x))^\alpha-x^\alpha e^{i(\xi/\sqrt{v_n}-\alpha \pi)}\}\{(t(1-x))^\alpha-x^\alpha e^{i(\xi/\sqrt{v_n}+\alpha \pi)}\}} \nn \\
        &\qquad+\frac{x^{\alpha (n+1)}(1-x)^{\alpha}e^{i\xi (n-2)/\sqrt{v_n}}}
        {\{(1-x)^\alpha e^{-i(\xi/\sqrt{v_n}+\alpha \pi)}-(tx)^\alpha\}
        \{(1-x)^\alpha e^{-i(\xi/\sqrt{v_n}-\alpha \pi)}-(tx)^\alpha\}} \Bigg\} \, \dd t\Bigg].
    \label{Eq:calculation-CF-1}
    \end{align}
    From Corollary \ref{Cor:Mean-variance-fractional-binomial}, we have
    \[
        \frac{m_n}{n}=x+O(e^{-\delta n}), 
        \qquad \sqrt{v_n}=\sqrt{\frac{1}{\alpha}nx(1-x)}+O(e^{-\delta n})=\Omega(\sqrt{n})
    \]
    as $n \to \infty$ for some $\delta>0$. Because 
        \[
        e^{i\theta}=1+i\theta-\frac{\theta^2}{2}+O(\theta^3)
        \]
    as $\theta \to 0$, we obtain 
    \begin{align}
        &\exp\left(-\frac{i\xi m_n}{\sqrt{v_n}}\right)
        \left(
        1-x+x\exp\left(\frac{i\xi}{\alpha\sqrt{v_n}}\right)
        \right)^{\alpha n}  \nn \\
        &=\left\{ (1-x)\exp\left(-\frac{i\xi m_n}{\alpha n \sqrt{v_n}}\right)
        +x \exp\left(\frac{i\xi }{\alpha \sqrt{v_n}}
        \left(1-\frac{m_n}{n}\right)\right)
        \right\}^{\alpha n} \nn \\
        &=\Bigg[ (1-x) \left\{1-\frac{i\xi m_n}{\alpha n \sqrt{v_n}}
        -\frac{\xi^2 m_n^2}{2(\alpha n \sqrt{v_n})^2}+O(n^{-3/2})\right\} \nn \\
        &\quad+x\left\{
        1+\frac{i\xi}{\alpha\sqrt{v_n}}\left(1-\frac{m_n}{n}\right)
        -\frac{\xi^2}{2(\alpha \sqrt{v_n})^2}\left(1-\frac{m_n}{n}\right)^2
        +O(n^{-3/2})\right\}\Bigg]^{\alpha n} \nn\\
        &= \Bigg[ (1-x) \left\{1-\frac{i\xi x}{\sqrt{\alpha n x(1-x)}}
        -\frac{\xi^2 x^2}{2 \alpha n x(1-x)}+O(n^{-3/2})\right\} \nn \\
        &\quad+x\left\{
        1+\frac{i\xi(1-x)}{\sqrt{\alpha n x(1-x)}} 
        -\frac{\xi^2(1-x)^2}{2\alpha n x(1-x)}
        +O(n^{-3/2})\right\}\Bigg]^{\alpha n} \nn\\
        &=\left\{ 
        1-\frac{\xi^2}{2\alpha n}+O(n^{-3/2})\right\}^{\alpha n}
        \to \exp\left(-\frac{\xi^2}{2}\right)
        \label{Eq:calculation-CF-2}
    \end{align}
    as $n \to \infty$. 
    For each $\omega \in K_\alpha \setminus \{1\}$, we also have
    \begin{align}
        &\left|\exp\left(-\frac{i\xi m_n}{\sqrt{v_n}}\right)
        \left(1-x+x\omega\exp\left(\frac{i\xi}{\alpha\sqrt{v_n}}\right)\right)^{\alpha n}\right| 
        =\left|1-x+x\omega+O(n^{-1/2})\right|^{\alpha n} \to 0
        \label{Eq:calculation-CF-3}
    \end{align}
    as $n \to \infty$, by noting \eqref{Eq:Estimate-1-x+xomega}. 
    To estimate the last term of \eqref{Eq:calculation-CF-1},
    we may assume $\alpha \notin \N$; otherwise this term vanishes. 
    Because it follows from Lemma \ref{Lem:simple-ineq} that 
    \begin{align*}
        &|(t(1-x))^\alpha-x^\alpha e^{i(\xi/\sqrt{v_n} \mp \alpha\pi)}| 
        \ge x^{\alpha}\left|\sin\frac{1}{2}\left(\frac{\xi}{\sqrt{v_n}} \mp \alpha \pi \right)\right| 
        \to x^{\alpha}\left|\sin \frac{\alpha \pi}{2}\right|>0
    \end{align*}
    and 
    \begin{align*}
        &|(1-x)^\alpha e^{-i(\xi/\sqrt{v_n} \pm \alpha\pi)} -(tx)^\alpha |
        \ge (1-x)^{\alpha}\left|\sin\frac{1}{2}\left(\frac{\xi}{\sqrt{v_n}} \pm \alpha \pi\right)  \right|
        \to (1-x)^{\alpha}\left|\sin \frac{\alpha \pi}{2}\right|>0
    \end{align*}
    as $n \to \infty$, Proposition~\ref{Prop:Order-moment-fractional-binomial}~(1) implies that 
    the third term of the right-hand side of \eqref{Eq:calculation-CF-1} converges to zero as $n \to \infty$.
    Therefore, \eqref{Eq:calculation-CF-1}, \eqref{Eq:calculation-CF-2},
    and \eqref{Eq:calculation-CF-3} lead to 
    \[
        \lim_{n \to \infty}\widetilde{\varphi}_{\alpha, x}^{(n)}(\xi)
        =\exp\left(-\frac{\xi^2}{2}\right), \qquad \xi \in \R.
    \]
    Because the right-hand side of the above is nothing but the
    characteristic function of the normal distribution $N(0, 1)$, 
    L\'evy's continuity theorem (cf.~\cite[Theorem~15.23]{Klenke}) 
    concludes the desired weak convergence. 
\end{proof}
By combining Corollary~\ref{Cor:Mean-variance-fractional-binomial} and Theorem~\ref{Thm:CLT}, we also have the following.
\begin{co}
In the notation of Theorem~\ref{Thm:CLT}, the law of $\frac{S_{\alpha, x}^{(n)} - nx}{\sqrt{nx(1-x)/\alpha}}$ converges weakly to $N(0, 1)$ as $n \to \infty$.
\end{co}

\subsection{The law of small numbers}
\label{Subsect:LSN}

Recall that the Mittag-Leffler function of parameter $\alpha>0$ is defined by 
    \[
        E_\alpha(z)=\sum_{j=0}^\infty
        \frac{z^j}{\Gamma(\alpha j+1)}, \qquad z \in \mathbb{C},
    \]
which is a special function appearing in the context of fractional calculus. 
In terms of this function, we introduce a fractional analogue 
of the Poisson distribution. 

\begin{df}[Fractional Poisson distribution]
\label{Def:fractional Poisson distribution}
    Let $\alpha >0$ and $\lambda>0$.  
    The $\alpha$-fractional Poisson distribution with parameter $\lambda$ is
    a probability measure $\nu_{\alpha, \lambda}$ 
    defined as 
    \[
        \nu_{\alpha, \lambda}(\dd z)
        :=\frac{1}{E_\alpha(\lambda^\alpha)} \sum_{j=0}^\infty \frac{\lambda^{\alpha j}}{\Gamma(\alpha j+1)}
        \delta_j(\dd z).
    \]
\end{df}

The $1$-fractional Poisson distribution coincides with 
the Poisson distribution. 
Some basic properties of this fractional Poisson distribution 
will be summarized in the subsequent section.
We note that $\nu_{\alpha, \lambda}$ is a special case of 
generalized fractional Poisson distributions defined via the generalized Mittag-Leffler function 
$E_{\alpha, \beta}(z)$ given by \eqref{Eq:generalized-ML-function} below.
See \cite{Herrmann, CO} for some properties as well as its definition. 
We also refer to \cite{Laskin} for the study of non-Markovian fractional Poisson processes 
based on Riemann--Liouville type fractional calculus, 
and \cite{Lee2} for some generalizations of Bernoulli processes having long-range dependence
and their convergences to fractional Poisson processes, 
both of which are different concepts from ours. 
We also mention that the Mittag-Leffler function has already been used to 
introduce some probability distributions in different situations. 
See \cite{Pillai, PJ} for the notion of the {\it Mittag-Leffler distribution},
which is, however, far from the Poisson distribution. 

Here, we prove that $\nu_{\alpha, \lambda}$ is captured 
through the law of small numbers for the fractional binomial distribution.

\begin{tm}[Law of small numbers for $\mu_{\alpha, x}^{(n)}$]
\label{Thm:LSN}
    Let $\alpha>0$ and $\lambda>0$. 
    Suppose that two sequences $\{\alpha_n\}_{n=1}^\infty \subset (0, \infty)$ and 
    $\{x_n\}_{n=1}^\infty \subset (0, 1)$ satisfy that 
    $\alpha_n \to \alpha$ and $n x_n \to \lambda/\alpha$ as $n \to \infty$. 
    Then, $\mu_{\alpha_n, x_n}^{(n)}$ converges weakly to $\nu_{\alpha, \lambda}$ as $n \to \infty$. 
\end{tm}

\begin{proof}
    The proof is split into two parts. 

    \smallskip
    \noindent
    {\bf Step 1.}
    We put 
        \[
            b_{\alpha_n, x_n}^{(n)}(j)=\alpha_n \binom{\alpha_n n}{\alpha_n j}x_n^{\alpha_n j}
            (1-x_n)^{\alpha_n(n-j)}\bm{1}_{[0, n]}(j), \qquad j \in \N_0.
        \]
    Because $\alpha_n \to \alpha>0$ and $n x_n \to \lambda/\alpha>0$, 
    there exist four constants $C_{3,1}, C_{3,2}, C_{3,3}, C_{3,4}>0$ such that 
        \begin{equation}\label{Eq:alpha_n-estimate}
            C_{3,1} \le \alpha_n \le C_{3,2}, \qquad 
            C_{3,3}  \le  n x_n \le C_{3,4}, \qquad n \in \N. 
        \end{equation}
    Moreover, the Stirling formula implies the existence of 
    two constants $C_{3,5}>0$ and $C_{3,6}>0$ satisfying 
        \begin{equation}\label{Eq:Stirling-estimate}
            C_{3,5}s^{s+1/2}e^{-s} \le \Gamma(s+1) \le C_{3,6} s^{s+1/2}e^{-s}, 
            \qquad s \ge \alpha >0. 
        \end{equation}
    We now give an estimate of each $b_{\alpha_n, x_n}^{(n)}(j)$.
    Let $j \in \N_0$ be fixed. If $j=0$, then we have 
        \[
            b_{\alpha_n, x_n}^{(n)}(j)
            =\alpha_n(1-x_n)^{\alpha_n n} 
            \le C_{3, 2}\left(1-\frac{C_{3,3}}{n}\right)^{C_{3,1}n}
            \le C_{3, 2}e^{-C_{3,1}C_{3,3}}.
        \]
    Suppose that $j \ge 1$. 
    If $n<j$, then $b_{\alpha_n, x_n}^{(n)}(j)=0$. 
    If $n=j$, then it follows from \eqref{Eq:alpha_n-estimate} that 
        \begin{align*}
            (0 \le)\,b_{\alpha_n, x_n}^{(n)}(j)
            &=\alpha_n x_n^{\alpha_n n} 
            \le C_{3, 2}x_n^{C_{3,1}n} \le C_{3, 2}\left(\frac{C_{3, 4}}{n}\right)^{C_{3, 1}n}=C_{3, 2}\left(\frac{C_{3, 4}}{j}\right)^{C_{3, 1}j}.
        \end{align*}
    On the other hand, if $n>j$, then \eqref{Eq:alpha_n-estimate} and \eqref{Eq:Stirling-estimate}
    allow us to obtain
        \begin{align*}
            (0 \le)\,b_{\alpha_n, x_n}^{(n)}(j)   
            &= \frac{\alpha_n\Gamma(\alpha_n n+1)}{\Gamma(\alpha_n(n-j)+1)\Gamma(\alpha_n j+1)}x_n^{\alpha_n j}(1-x_n)^{\alpha_n(n-j)} \nn \\
            &\le  \frac{C_{3, 2}C_{3,6}(\alpha_n n)^{\alpha_n n+1/2}e^{-\alpha_n n}}{C_{3,5}\{\alpha_n(n-j)\}^{\alpha_n(n-j)+1/2}e^{-\alpha_n(n-j)} \Gamma(\alpha_n j+1)}
            \left(\frac{C_{3, 4}}{n}\right)^{\alpha_n j}\nn \\
            &\le \frac{C_{3,2}C_{3,6}(C_{3,4}\alpha_n)^{\alpha_n j}}{C_{3,5}\Gamma(C_{3, 1} j+1)}e^{-\alpha_n j}
            \left(\frac{n}{n-j}\right)^{\alpha_n(n-j)+1/2}. 
        \end{align*}
    By noting that the function $(0, \infty) \ni s \longmapsto (1+s)^{1/s} \in (0, \infty)$
    is non-increasing, we have 
        \[
            \left(\frac{n}{n-j}\right)^{\alpha_n(n-j)}
            =\left\{ \left(1+\frac{j}{n-j}\right)^{(n-j)/j}\right\}^{\alpha_n j} \le e^{\alpha_n j} .
        \]
    Moreover, we have 
        \[
            \left(\frac{n}{n-j}\right)^{1/2} \le \left(\frac{j+1}{j+1-j}\right)^{1/2}=(j+1)^{1/2}.
        \]
    Bringing everything together, we have
        \begin{align}\label{Eq:b-estimate-3}
            (0 \le)\,b_{\alpha_n, x_n}^{(n)}(j)  
            &=\begin{cases}
            \dis C_{3,2}e^{-C_{3,1}C_{3,3}}
            & \text{if }j=0, \smallskip\\
            \dis \frac{C_{3,2} C_{3,6}(C_{3,2}C_{3,4} \vee 1)^{C_{3,2} j}}{C_{3,5}\Gamma(C_{3,1} j+1)}(j+1)^{1/2}  & \text{if }j=1, 2, \dots, n-1, \smallskip\\
            \dis C_{3,2}  \left(\frac{C_{3,4}}{j}\right)^{C_{3,1} j}
            & \text{if }j=n, \\ 
            0 & \text{if }j>n
            \end{cases} \nn \\
            &\le C_{3,7} \left\{\frac{C_{3,8}^{C_{3,2} j}}{\Gamma(C_{3,1} j +1)} \vee \left(\frac{C_{3, 4}}{j}\right)^{C_{3, 1}j}\right\}
        \end{align}
    for some constants $C_{3,7}, C_{3,8}>0$. 
    Hence, \eqref{Eq:b-estimate-3} leads to
        \begin{equation}\label{Eq:LSN-sup}
            \sum_{j=0}^\infty \sup_{n \in \N} b_{\alpha_n, x_n}^{(n)}(j)
            < \infty.
        \end{equation}
    
    \smallskip
    \noindent{\bf Step 2.}
    By applying the Stirling formula, we have 
        \begin{align*}
            \binom{\alpha_n n}{\alpha_n j}x_n^{\alpha_n j}
            &= \frac{\Gamma(\alpha_n n+1)}{\Gamma(\alpha_n j +1)\Gamma(\alpha_n(n-j)+1)}x_n^{\alpha_n j}\\
            &\sim \frac{\sqrt{2\pi}(\alpha_n n)^{\alpha_n n+1/2}e^{-\alpha_n n}}
            {\Gamma(\alpha_n j+1)\sqrt{2\pi}\{\alpha_n (n-j)\}^{\alpha_n(n-j)+1/2}e^{-\alpha_n(n-j)}} 
            \left(\frac{\lambda}{\alpha_n n}\right)^{\alpha_n j}\\
            &=\frac{\lambda^{\alpha_n j}}{\Gamma(\alpha_n j+1)}e^{-\alpha_n j} 
            \left\{\left(1-\frac{j}{n}\right)^n\right\}^{-\alpha_n(1-j/n)-1/(2n)}\\
            &\to \frac{ \lambda^{\alpha j}}{\Gamma(\alpha j+1)}e^{-\alpha j} 
            e^{\alpha j}=\frac{\lambda^{\alpha j}}{\Gamma(\alpha j+1)} 
        \end{align*}
    as $n \to \infty$. Therefore, for $j \in \N_0$, we obtain 
        \begin{align*}
            b_{\alpha_n, x_n}^{(n)}(j)
            &=  \alpha_n \binom{\alpha_n n}{\alpha_n j}x_n^{\alpha_n j}(1-x_n)^{\alpha_n(n-j)}
            \bm{1}_{[0, n]}(j)\\
            &= \alpha_n \binom{\alpha_n n}{\alpha_n j}x_n^{\alpha_n j}\left(1-\frac{1}{\alpha_n n}(\lambda+o(1))\right)^{\alpha_n(n-j)} \bm{1}_{[0, n]}(j) \\
            &\to \frac{\alpha \lambda^{\alpha j}}{\Gamma(\alpha j+1)} e^{- \lambda}
        \end{align*}
    as $n \to \infty$. Because 
        \[
            Z_{\alpha_n, x_n}^{(n)}
            =\sum_{j=0}^n \alpha_n \binom{\alpha_n n}{\alpha_n j}
            x_n^{\alpha_n j}(1-x_n)^{\alpha_n(n-j)}
            =\sum_{j=0}^\infty b_{\alpha_n, x_n}^{(n)}(j),
        \]
    \eqref{Eq:LSN-sup} and the dominated convergence theorem imply that
        \[
            \lim_{n \to \infty}Z_{\alpha_n, x_n}^{(n)}
            =\sum_{j=0}^\infty \left( \lim_{n \to \infty} b_{\alpha_n, x_n}^{(n)}(j)\right) 
            =\sum_{j=0}^\infty\frac{\alpha \lambda^{\alpha j}}{\Gamma(\alpha j+1)} e^{-\lambda}=\alpha e^{-\lambda} E_\alpha(\lambda^\alpha).
        \]
    Thus, we obtain 
        \[
            \mu_{\alpha_n, x_n}^{(n)}(\{j\})
            =\frac{1}{Z_{\alpha_n, x_n}^{(n)}} 
            b_{\alpha_n, x_n}^{(n)}(j) 
            \to \frac{1}{\alpha e^{-\lambda} E_\alpha(\lambda^\alpha)}
            \frac{\alpha \lambda^{\alpha j}}{\Gamma(\alpha j+1)} e^{- \lambda}=\nu_{\alpha, \lambda}(\{j\})
        \]
    as $n \to \infty$ for all $j \in \N_0$. 
\end{proof}

\subsection{Some properties of the fractional Poisson distribution}
\label{Subsect:fractional Poisson distribution}

First, we seek the mean and variance of the 
fractional Poisson distribution $\nu_{\alpha, \lambda}$
introduced in Definition~\ref{Def:fractional Poisson distribution}. 

\begin{lm}\label{Lem:mean and variance of fractional Poisson}
    Let $S_{\alpha, \lambda}$ be a random variable whose law is $\nu_{\alpha, \lambda}$. 
    Then, we have 
        \begin{align*}
            \E[S_{\alpha, \lambda}]
            &=\frac{\lambda^\alpha E_{\alpha, \alpha}(\lambda^\alpha)}{\alpha E_\alpha(\lambda^\alpha)}, \\ 
            \mathrm{Var}(S_{\alpha, \lambda})
            &=\frac{\lambda^\alpha}{\alpha^2 E_\alpha(\lambda^\alpha)^2}\Big[
            \alpha E_{\alpha}(\lambda^\alpha)
            E_{\alpha, \alpha}(\lambda^\alpha)\\
            &\quad
            +\lambda^\alpha E_{\alpha}(\lambda^\alpha)
            \big\{ E_{\alpha, 2\alpha-1}(\lambda^\alpha)
            -(\alpha-1)E_{\alpha, 2\alpha}(\lambda^\alpha)\big\}
            -\lambda^\alpha
            E_{\alpha, \alpha}(\lambda^\alpha)^2\Big],
        \end{align*}
    where in general  $E_{\alpha, \beta}(z)$ is the Mittag-Leffler function 
    of two parameters $\alpha>0$ and $\beta \in \mathbb{C}$ defined by 
        \begin{equation}\label{Eq:generalized-ML-function}
        E_{\alpha, \beta}(z):=\sum_{j=0}^\infty \frac{z^j}{\Gamma(\alpha j+\beta)}, 
        \qquad z \in \mathbb{C}. 
        \end{equation}
\end{lm}

\begin{proof}
    Let $\varphi_{\alpha, \lambda}(\xi)$, $\xi \in \R$ be 
    the characteristic function of $\nu_{\alpha, \lambda}$. 
    It is straightforward to check that 
    \begin{align*}
        \varphi_{\alpha, \lambda}(\xi)
        &=\frac{1}{E_\alpha(\lambda^\alpha)}\sum_{j=0}^\infty e^{i\xi j} 
        \frac{\lambda^{\alpha j}}{\Gamma(\alpha j+1)} 
        =\frac{1}{E_\alpha(\lambda^\alpha)}
        \sum_{j=0}^\infty 
        \frac{(\lambda^\alpha e^{i\xi})^j}{\Gamma(\alpha j+1)} 
        =\frac{E_\alpha(\lambda^\alpha e^{i\xi})}{E_\alpha(\lambda^\alpha)}
    \end{align*}
    for $\xi \in \R$. Because  
    \begin{align}
        \frac{\dd}{\dd z}E_\alpha(z)
        &= \sum_{j=1}^\infty j \frac{z^{j-1}}{\Gamma(\alpha j+1)}
        =\frac{1}{\alpha}\sum_{j=0}^\infty \frac{z^j}{\Gamma(\alpha j+\alpha)}
        =\frac{1}{\alpha}E_{\alpha, \alpha}(z)
        \label{Eq:derivative-ML}
    \end{align}
    and
    \begin{align}
        \frac{\dd}{\dd z}E_{\alpha, \alpha}(z)
        &= \sum_{j=1}^\infty j  \frac{z^{j-1}}{\Gamma(\alpha j+\alpha)} \nn \\
        &= \sum_{j=0}^\infty (j+1)  \frac{z^j}{\Gamma(\alpha j+2\alpha)}\nn  \\
        &=\frac{1}{\alpha}\sum_{j=0}^\infty 
        (\alpha j+2\alpha -1)  \frac{z^j}{\Gamma(\alpha j+2\alpha)}
        -\frac{\alpha-1}{\alpha}\sum_{j=0}^\infty 
        \frac{z^j}{\Gamma(\alpha j + 2\alpha)}\nn \\
        &=\frac{1}{\alpha}\sum_{j=0}^\infty 
        \frac{z^j}{\Gamma(\alpha j+2\alpha-1)}
        -\frac{\alpha-1}{\alpha}\sum_{j=0}^\infty 
        \frac{z^j}{\Gamma(\alpha j + 2\alpha)}\nn \\
        &=\frac{1}{\alpha}E_{\alpha, 2\alpha-1}(z)-\frac{\alpha-1}{\alpha}E_{\alpha, 2\alpha}(z), 
        \label{Eq:2nd-derivative-ML}
    \end{align}
    we have 
    \begin{align}
        \frac{\dd}{\dd \xi}\varphi_{\alpha, \lambda}(\xi)
        &= \frac{1}{E_\alpha(\lambda^\alpha)}
        \frac{1}{\alpha}E_{\alpha, \alpha}(\lambda^\alpha e^{i\xi})
        i\lambda^\alpha e^{i\xi} 
        =i \frac{\lambda^\alpha e^{i\xi} E_{\alpha, \alpha}(\lambda^\alpha e^{i\xi})}{\alpha E_\alpha(\lambda^\alpha)}
        \label{Eq:fractional Poisson CF-1}
    \end{align}
    and 
    \begin{align}
        &\frac{\dd^2}{\dd \xi^2}\varphi_{\alpha, \lambda}(\xi)\nn\\
        &=i \frac{\lambda^\alpha}{\alpha E_\alpha(\lambda^\alpha)}
        \Bigg\{ ie^{i\xi}E_{\alpha, \alpha}(\lambda^\alpha e^{i\xi}) 
        +i\lambda^\alpha e^{2i\xi}\left(
        \frac{1}{\alpha}E_{\alpha, 2\alpha-1}(\lambda^\alpha e^{i\xi})
        -\frac{\alpha-1}{\alpha}
        E_{\alpha, 2\alpha}(\lambda^\alpha e^{i\xi})\right)\Bigg\} \nn \\
        &=- \frac{\lambda^\alpha e^{i\xi}}{\alpha E_\alpha(\lambda^\alpha)}
        \Bigg\{E_{\alpha, \alpha}(\lambda^\alpha e^{i\xi}) 
        +\lambda^\alpha e^{i\xi}
        \left(
        \frac{1}{\alpha}E_{\alpha, 2\alpha-1}(\lambda^\alpha e^{i\xi})-\frac{\alpha-1}{\alpha}
        E_{\alpha, 2\alpha}(\lambda^\alpha e^{i\xi})\right)\Bigg\}.
        \label{Eq:fractional Poisson CF-2}
    \end{align}
    By substituting $\xi=0$ for \eqref{Eq:fractional Poisson CF-1}
    and \eqref{Eq:fractional Poisson CF-2}, we obtain
    \begin{align*}
        \E[S_{\alpha, \lambda}]
        &=\frac{\lambda^\alpha E_{\alpha, \alpha}(\lambda^\alpha)}{\alpha E_\alpha(\lambda^\alpha)}, \\
        \E[S_{\alpha, \lambda}^2]
        &= \frac{\lambda^\alpha}{\alpha^2 E_\alpha(\lambda^\alpha)}\big[\alpha E_{\alpha, \alpha}(\lambda^\alpha) 
        +\lambda^\alpha \big\{
        E_{\alpha, 2\alpha-1}(\lambda^\alpha)
        -(\alpha-1)E_{\alpha, 2\alpha}(\lambda^\alpha)\big\}\big].
    \end{align*}
    The variance of $S_{\alpha, \lambda}$ is obtained straightforwardly from the two equalities above. 
\end{proof}

As with the Poisson distribution, it is quite cumbersome to express both the raw and central moments of the fractional Poisson distribution, but we can derive the following recurrence formula for its central moments. 

\begin{tm}
    For $m \in  \N_0$, we define 
        \begin{equation}\label{Eq:definition-moment-fP}
            M_{\alpha, \lambda}^{(m)}
            :=\E\big[(S_{\alpha, \lambda}-\E[S_{\alpha, \lambda}])^m\big]
            =\sum_{j=0}^{\infty} \left(j -
            \frac{\lambda^\alpha E_{\alpha, \alpha}(\lambda^\alpha)}{\alpha E_\alpha(\lambda^\alpha)}\right)^m
            \nu_{\alpha, \lambda}(\{j\}). 
        \end{equation}
    Then, we have 
        \begin{align}
            M_{\alpha, \lambda}^{(m+1)}&=
            \frac{\lambda}{\alpha}\frac{\dd}{\dd \lambda}M_{\alpha, \lambda}^{(m)}
            +mM_{\alpha, \lambda}^{(2)}M_{\alpha, \lambda}^{(m-1)}. 
            \label{Eq:recurrence formula fractional Poisson}
        \end{align}
\end{tm}

\begin{proof}
    Thanks to \eqref{Eq:derivative-ML} and \eqref{Eq:2nd-derivative-ML}, we have 
    \begin{align*}
        \frac{\dd}{\dd \lambda}E_\alpha(\lambda^\alpha)
        &=\lambda^{\alpha-1} E_{\alpha, \alpha}(\lambda^\alpha), \\
        \frac{\dd}{\dd \lambda}E_{\alpha, \alpha}(\lambda^\alpha)
        &=\lambda^{\alpha-1}
        \left\{ E_{\alpha, 2\alpha-1}(\lambda^\alpha)
        -(\alpha-1)E_{\alpha, 2\alpha}(\lambda^\alpha)\right\}. 
    \end{align*}
    Moreover, we have
    \begin{align*}
        \frac{\dd}{\dd \lambda}\nu_{\alpha, \lambda}(\{j\})
        &=-\frac{\lambda^{\alpha-1}E_{\alpha, \alpha}(\lambda^\alpha)}{E_\alpha(\lambda^\alpha)^2}
        \frac{\lambda^{\alpha j}}{\Gamma (\alpha j+1)}
        +\frac{\alpha j}{\lambda E_\alpha(\lambda^\alpha)}\frac{\lambda^{\alpha j}}{\Gamma(\alpha j+1)}\\
        &=\frac{\alpha}{\lambda E_\alpha(\lambda^\alpha)}
        \left(j -\frac{\lambda^\alpha E_{\alpha, \alpha}(\lambda^\alpha)}{\alpha E_\alpha(\lambda^\alpha)}\right)
        \frac{\lambda^{\alpha j}}{\Gamma(\alpha j+1)}\\
        &=\frac{\alpha}{\lambda}\left(j -\frac{\lambda^\alpha E_{\alpha, \alpha}(\lambda^\alpha)}{\alpha E_\alpha(\lambda^\alpha)}\right)\nu_{\alpha, \lambda}(\{j\})
    \end{align*}
    for $j \in \N_0$. 
    By differentiating both sides of \eqref{Eq:definition-moment-fP} with respect to $\lambda$, 
    we obtain
    \begin{align*}
        \frac{\dd}{\dd \lambda}M_{\alpha, \lambda}^{(m)}
        &=-m \sum_{j=0}^\infty \left(j -
        \frac{\lambda^\alpha E_{\alpha, \alpha}(\lambda^\alpha)}{\alpha E_\alpha(\lambda^\alpha)}\right)^{m-1} \nu_{\alpha, \lambda}(\{j\})
        \frac{\dd}{\dd \lambda}\left(\frac{\lambda^\alpha E_{\alpha, \alpha}(\lambda^\alpha)}{\alpha E_\alpha(\lambda^\alpha)}\right)\\
        &\quad+\sum_{j=0}^\infty 
        \left(j -
        \frac{\lambda^\alpha E_{\alpha, \alpha}(\lambda^\alpha)}{\alpha E_\alpha(\lambda^\alpha)}\right)^m
        \frac{\alpha}{\lambda}\left(j -\frac{\lambda^\alpha E_{\alpha, \alpha}(\lambda^\alpha)}{\alpha E_\alpha(\lambda^\alpha)}\right)\nu_{\alpha, \lambda}(\{j\}) \\\
        &=-mM_{\alpha, \lambda}^{(m-1)}\frac{\dd}{\dd \lambda}\left(\frac{\lambda^\alpha E_{\alpha, \alpha}(\lambda^\alpha)}{\alpha E_\alpha(\lambda^\alpha)}\right)
        +\frac{\alpha}{\lambda}M_{\alpha, \lambda}^{(m+1)}.
    \end{align*}
    Because it holds that 
    \begin{align*}
      \frac{\dd}{\dd \lambda}\left(\frac{\lambda^\alpha E_{\alpha, \alpha}(\lambda^\alpha)}{\alpha E_\alpha(\lambda^\alpha)}\right)
      &=\frac{\lambda^{\alpha-1} }{\alpha E_\alpha(\lambda^\alpha)^2}\Big[
        \alpha E_{\alpha}(\lambda^\alpha)
        E_{\alpha, \alpha}(\lambda^\alpha)\\
        &\quad
        + \lambda^\alpha E_{\alpha}(\lambda^\alpha)
        \big\{
        E_{\alpha, 2\alpha-1}(\lambda^\alpha)
        -(\alpha-1)E_{\alpha, 2\alpha}(\lambda^\alpha)\big\}
        - \lambda^\alpha 
        E_{\alpha, \alpha}(\lambda^\alpha)^2\Big]\\
        &=\frac{\alpha}{\lambda}\mathrm{Var}(S_{\alpha, \lambda})
        =\frac{\alpha}{\lambda}M_{\alpha, \lambda}^{(2)},
    \end{align*}
    we conclude \eqref{Eq:recurrence formula fractional Poisson}. 
\end{proof}

\section{{\bf The fractional Bernstein operator and limit theorems for its iterates}}
\label{Sect:fractional Bernstein operator}

\subsection{The fractional Bernstein operator and its approximating property}
\label{Subsect:fractional Bernstein}

Let $\alpha>0$. In this section, we prove that the fractional Bernstein operator 
$B_{\alpha, n}$ introduced in Definition~\ref{Def:fractional-Bernstein-operator} 
uniformly approximates all continuous functions on $[0, 1]$,
similarly to the classical case. 

\begin{tm}
\label{Thm:Fractional-Bernstein-approximation}
    Let $K$ be a compact subset of $(0, \infty)$. 
    For $f \in C([0, 1])$, we have 
        \[
            \lim_{n \to \infty}
            \sup_{\alpha \in K}\|B_{\alpha, n}f-f\|_\infty=0. 
        \] 
\end{tm}

\begin{proof}
    Let $\alpha>0$, $n \in \N$, $x \in [0, 1]$, and $S_{\alpha, x}^{(n)}$ be a random variable whose 
    law is $\mu_{\alpha, x}^{(n)}$. Then, 
    $B_{\alpha, n}f$, $f \in C([0, 1])$ can be expressed as 
        \begin{equation}\label{Eq:fractional-Bernstein-expression}
            B_{\alpha, n}f(x)
            =\E\left[f \left(\frac{1}{n}S_{\alpha, x}^{(n)}\right)\right], 
            \qquad x \in [0, 1]. 
        \end{equation}
    Because $f \in C([0, 1])$ is uniformly continuous on $[0, 1]$, for $\ve>0$, 
    there exists some $\rho=\rho(\ve)>0$ such that $|x-y|<\rho$ implies that $|f(x)-f(y)|<\ve$. 
    Then, it follows from \eqref{Eq:fractional-Bernstein-expression} and the Chebyshev inequality that
        \begin{align*}
            |B_{\alpha, n}f(x)-f(x)|
            &= \left|\E\left[f\left(\frac{1}{n}S_{\alpha, x}^{(n)}\right)-f(x)\right]\right| \\
            &\le \E\left[\left|f\left(\frac{1}{n}S_{\alpha, x}^{(n)}\right)-f(x)\right|
            \, : \, \left|\frac{1}{n}S_{\alpha, x}^{(n)}-x\right|<\rho\right] \\
            &\quad
            +\E\left[\left|f\left(\frac{1}{n}S_{\alpha, x}^{(n)}\right)-f(x)\right|
            \, : \, \left|\frac{1}{n}S_{\alpha, x}^{(n)}-x\right| \ge \rho\right] \\
            &\le \ve 
            + 2\|f\|_\infty \mathbb{P}\left(\left|\frac{1}{n}S_{\alpha, x}^{(n)}-x\right| \ge \rho\right). 
        \end{align*}
    Therefore, we obtain from Theorem~\ref{Thm:WLLN} that 
        \[
            \sup_{\alpha \in K}\|B_{\alpha, n}f-f\|_\infty
            =\ve + O\left(\frac{1}{n}\right)
        \]
    as $n \to \infty$, which concludes the desired uniform convergence because $\ve>0$ is chosen arbitrarily. 
\end{proof}

\subsection{The Kelisky--Rivlin-type theorem for the fractional Bernstein operator}
\label{Subsect:Kelisky-Rivlin}

It is natural to consider the behaviors of the $k$-fold iterates $B_{\alpha, n}^k$ 
of the $\alpha$-fractional Bernstein operator $B_{\alpha, n}$ as $k \to \infty$. 
In this section, we prove the Kelisky--Rivlin-type limit theorem 
for $B_{\alpha, n}$ 
based on an argument similar to that in the proof of \cite[Theorem~1]{Rus}.

\begin{tm}
\label{Thm:Kelisky-Rivlin-fractional}
    Let $\alpha > 0$ and $n \in \N$. 
    For $f \in C([0, 1])$, there exists a unique function 
    $f_{\alpha, n} \in C([0, 1])$ satisfying 
    $B_{\alpha, n}f_{\alpha, n}=f_{\alpha, n}$, $f_{\alpha, n}(0)=f(0)$, $f_{\alpha, n}(1)=f(1)$, 
    and 
        \begin{equation}\label{Eq:Kelisky-Rivlin-fractional}
            \lim_{k \to \infty} \|B_{\alpha, n}^k f - f_{\alpha, n}\|_\infty=0.
        \end{equation}
    In particular, we have 
        \[
            f_{1, n}(x)=f(1)x+f(0)(1-x), \qquad x \in [0, 1]. 
        \]
\end{tm}

\begin{proof}
    Let $\alpha>0$ and $n \in \N$. 
    For $a, b \in \R$, we define a closed subset $C_{a, b}([0, 1])$ of $C([0, 1])$ by 
        \[
            C_{a, b}([0, 1]):=\{f \in C([0, 1]) \mid f(0)=a, f(1)=b\}. 
        \]
    Then, we see that $f \in C_{a, b}([0, 1])$ implies $B_{\alpha, n}f \in C_{a, b}([0, 1])$ 
    because $B_{\alpha, n}f(0)=f(0)$ and $B_{\alpha, n}f(1)=f(1)$.
    For $f, g \in C_{a, b}([0, 1])$, we have 
        \begin{align*}
            |B_{\alpha, n}f(x)-B_{\alpha, n}g(x)| 
            &=\left|\sum_{j=0}^n \mu_{\alpha, x}^{(n)}(\{j\}) \left\{ f\left(\frac{j}{n}\right)-g\left(\frac{j}{n}\right)\right\}\right|\\
            &=\left|\sum_{j=1}^{n-1} \mu_{\alpha, x}^{(n)}(\{j\}) \left\{ f\left(\frac{j}{n}\right)-g\left(\frac{j}{n}\right)\right\}\right| \\
            &\le \|f-g\|_\infty 
            \left|\sum_{j=0}^{n} \mu_{\alpha, x}^{(n)}(\{j\}) - \frac{\alpha x^{\alpha n}}{Z_{\alpha, x}^{(n)}}
            - \frac{\alpha (1-x)^{\alpha n}}{Z_{\alpha, x}^{(n)}}\right| \\
            &=\|f-g\|_\infty \left(1-\frac{\alpha}{Z_{\alpha, x}^{(n)}}\{x^{\alpha n}+(1-x)^{\alpha n}\}\right).
        \end{align*}
    We note that 
        \[
            x^{\alpha n}+(1-x)^{\alpha n} 
            \ge \left(\frac{1}{2}\right)^{\alpha n}, \qquad x \in [0, 1]. 
        \]
    Moreover, we have 
        \[
            \frac{\alpha}{Z_{\alpha, x}^{(n)}} \ge 1 \wedge \alpha, \qquad x \in [0, 1]
        \]
    by applying Theorem~\ref{Thm:Uniform-bound-normalized-constant}. Therefore, we obtain 
        \[
            1-\frac{\alpha}{Z_{\alpha, x}^{(n)}}\{x^{\alpha n}+(1-x)^{\alpha n}\}
            \le 1-(1 \wedge \alpha)\left(\frac{1}{2}\right)^{\alpha n}<1, \qquad x \in [0, 1], 
        \]
    which means that $B_{\alpha, n}$ is contractive on $C_{a, b}([0, 1])$ for $a, b \in \R$. 
    Then, the contraction principle allows us to find the unique function 
    $f_{\alpha, n} \in C([0, 1])$ satisfying 
    $B_{\alpha, n}f_{\alpha, n}=f_{\alpha, n}$, $f_{\alpha, n}(0)=f(0)$, 
    $f_{\alpha, n}(1)=f(1)$, and \eqref{Eq:Kelisky-Rivlin-fractional}. 
    The explicit representation of $f_{1, n}$ is obvious from \cite[Theorem~1]{KR67}. 
\end{proof}

\subsection{Limit theorems for iterates of the fractional Bernstein operator}
\label{Subsect:limit theorems for fractional Bernstein}

Let $\alpha >0$ be fixed. 
In this section, we identify some limiting behaviors of the iterate
of $B_{\alpha, n}$. 
Let $\mathcal{L}_\alpha$ be the second-order differential operator 
acting on $C^2([0, 1])$ defined by 
    \[
        \mathcal{L}_\alpha f(x):=\frac{1}{2\alpha}x(1-x)f''(x)
    \]
for $f \in C^2([0,1])$ and $x \in [0, 1]$. 
The following property is well known. 

\begin{lm}[cf.~{\cite[Theorem~8.2.8]{EK}} and {\cite[Theorem~6.2.6]{Alt-C}}]
\label{Lem:Altomare-Campiti}
    The linear operator $(\mathcal{L}_\alpha, C^2([0, 1]))$ is closable on $C([0, 1])$ 
    and its closure $(\mathcal{L}_\alpha, \mathrm{Dom}(\mathcal{L}_\alpha))$
    generates the Feller semigroup $(T_\alpha(t))_{t \ge 0}$ on $C([0, 1])$.
\end{lm}

We prove that the appropriately scaled operators $(B_{\alpha, n}^{\lfloor nt \rfloor})_{n \in \N_0}$, $t \ge 0$, defined by the iterates of $B_{\alpha,n}$ converges to
$(T_\alpha(t))_{t \ge 0}$ as follows. 

\begin{tm}
\label{Thm:convergence-fractional-Bernstein-1}
    Let $\alpha>0$. Then, we have 
        \begin{equation}\label{Eq:semigroup-uniform-convergence}
            \lim_{n \to \infty} \|B_{\alpha,n}^{\lfloor nt \rfloor}f
            - T_\alpha(t)f\|_\infty=0, \qquad f \in C([0, 1]), \, t \ge 0.
        \end{equation}
\end{tm} 
    
In order to prove Theorem~\ref{Thm:convergence-fractional-Bernstein-1}, 
we prove the \textit{Voronovskaya-type theorem} which asserts the convergence 
of the sequence $\{n(B_{\alpha, n} -I)\}_{n=1}^\infty$ to $\mathcal{L}_\alpha$, 
where $I$ denotes the identity operator on $C([0,1])$. 
Let us put
    \[
        \mathcal{D}_0 := \{f \in C^2([0, 1]) \mid f'(0)=f'(1)=0\}.
    \]

\begin{lm}\label{Lem:Voronovskaya-Bernstein}
    Let $K$ be a compact subset of $(0, \infty)$. 
    Then, we have 
    \begin{equation}\label{Eq:Voronovskaya-Bernstein}
        \lim_{n \to \infty} \sup_{\alpha \in K}\| n(B_{\alpha, n} - I)f - \mathcal{L}_\alpha f\|_\infty=0,
        \qquad f \in \mathcal{D}_0. 
    \end{equation}
\end{lm}

\begin{proof}
    Let $f \in \mathcal{D}_0$. 
    We apply \eqref{Eq:fractional-Bernstein-expression} and Taylor's formula to obtain
    \begin{align}
        &n(B_{\alpha, n} - I)f(x)-\mathcal{L}_\alpha f(x) \nn \\
        &=n \E\left[f\left(\frac{1}{n}S_{\alpha, x}^{(n)}\right)-f(x)\right]
        -\mathcal{L}_\alpha f(x) \nn \\
        &= nf'(x)\E\left[ \frac{1}{n}S_{\alpha, x}^{(n)}-x \right] \nn \\
        &\quad+n \E\Bigg[\int_0^1 (1-t)
        \left\{f''\left(x+t\left(\frac{1}{n}S_{\alpha, x}^{(n)}-x\right)\right)-f''(x) \right\}
        \left(\frac{1}{n}S_{\alpha, x}^{(n)}-x\right)^2 \, \dd t\Bigg] \nn \\
        &\quad
        +\frac{n}{2Z_{\alpha, x}^{(n)}}
        \left\{\alpha \sum_{j=0}^n \binom{\alpha n}{\alpha j}
        x^{\alpha j}(1-x)^{\alpha(n-j)}\left(\frac{j}{n}-x\right)^2 - \frac{1}{\alpha n}x(1-x)\right\}f''(x) \nn \\
        &\quad+\frac{1}{2\alpha}
        \left(\frac{1}{Z_{\alpha, x}^{(n)}}-1\right)x(1-x)f''(x) \nn \\
        &=: \I_{\alpha}^{(n)}(x)+\II_{\alpha}^{(n)}(x)
        +\III_\alpha^{(n)}(x)
        + \IV_\alpha^{(n)}(x).
        \label{Eq:difference of generators-0}
    \end{align}
    Let $\ve>0$. 
    By assumption, there exists some $r>0$ such that 
    \[
        |f'(x)|<\ve, \qquad x \in [0, r] \cup [1-r, 1]. 
    \]
    Then, it follows from Proposition~\ref{Prop:Order-moment-fractional-binomial},
    Theorem~\ref{Thm:Uniform-bound-normalized-constant},
    and Corollary~\ref{Cor:Moment-estimate} that 
    \begin{align}\label{Eq:difference of generators-1}
        \left| \I_\alpha^{(n)}(x)\right| 
        &\le n\max_{x \in [0, r] \cup [1-r, 1]}
        \left|\frac{\alpha}{Z_{\alpha, x}^{(n)}}  
        \sum_{j=0}^n \binom{\alpha n}{\alpha j}
        x^{\alpha j}(1-x)^{\alpha(n-j)} 
        \left(\frac{j}{n}-x\right)\right||f'(x)|  \nn \\
        &\quad+n\|f'\|_\infty\max_{x \in [r, 1-r]}
        \left|\frac{\alpha}{Z_{\alpha, x}^{(n)}} 
        \sum_{j=0}^n \binom{\alpha n}{\alpha j}
        x^{\alpha j}(1-x)^{\alpha(n-j)} 
        \left(\frac{j}{n}-x\right)\right| \nn \\
        &\le n \left(\sup_{n \in \N}\sup_{\alpha \in K}\max_{x \in [0, 1]}
        \frac{1}{Z_{\alpha, x}^{(n)}}\right) 
        \sup_{\alpha \in K}\max_{x \in [0, 1]}
        \left|\alpha\sum_{j=0}^n \binom{\alpha n}{\alpha j}
        x^{\alpha j}(1-x)^{\alpha(n-j)} 
        \left(\frac{j}{n}-x\right)\right| \ve \nn \\
        &\quad+n\|f'\|_\infty  \left(\sup_{n \in \N}\sup_{\alpha \in K}\max_{x \in [0, 1]}\frac{1}{Z_{\alpha, x}^{(n)}}\right)
        O(e^{-\delta n}) \nn \\
        &= \ve O(1) + n O(e^{-\delta n}) 
        = \ve O(1)+o(1)
    \end{align}
    as $n \to \infty$.
    Because $f''$ is uniformly continuous on $[0, 1]$, 
    there exists some $\rho=\rho(\ve) \in (0, 1)$ such that 
    $|x-y|<\rho$ implies $|f''(x) - f''(y)|<\ve$. 
    Hence, Theorem~\ref{Thm:Uniform-bound-normalized-constant} and Corollary~\ref{Cor:Moment-estimate} 
    allow us to obtain that
    \begin{align}
        &\left|\II_\alpha^{(n)}(x)\right| \nn \\
        &\le \frac{n\ve}{2}\max_{x \in [0, 1]}\E\left[\left(\frac{1}{n}S_{\alpha, x}^{(n)}-x\right)^2\right] 
        +\frac{n\|f''\|_\infty}{\rho^2}
        \max_{x \in [0, 1]}\E\left[\left(\frac{1}{n}S_{\alpha, x}^{(n)}-x\right)^4\right] \nn \\ 
        &\le \frac{n\ve}{2}  \left(\sup_{n \in \N}\sup_{\alpha \in K}\max_{x \in [0, 1]}
        \frac{1}{Z_{\alpha, x}^{(n)}}\right)
        \sup_{\alpha \in K}\max_{x \in [0, 1]}
        \left| \alpha \sum_{j=0}^n \binom{\alpha n}{\alpha j}
        x^{\alpha j}(1-x)^{\alpha(n-j)}\left(\frac{j}{n}-x\right)^2\right| \nn \\
        &\quad
        +\frac{n\|f''\|_\infty}{\rho^2} \left(\sup_{n \in \N}\sup_{\alpha \in K}\max_{x \in [0, 1]}
        \frac{1}{Z_{\alpha, x}^{(n)}}\right)
        \sup_{\alpha \in K}\max_{x \in [0, 1]}
        \left| \alpha \sum_{j=0}^n \binom{\alpha n}{\alpha j}
        x^{\alpha j}(1-x)^{\alpha(n-j)}\left(\frac{j}{n}-x\right)^4\right| \nn \\
        &= n\ve O\left(\frac{1}{n}\right)
        +\frac{n\|f''\|_\infty}{\rho^2} O\left(\frac{1}{n^2}\right)  \nn \\
        &=\ve O(1)+O\left(\frac{1}{n}\right)
        \label{Eq:difference of generators-2}
    \end{align}
    as $n \to \infty$. 
    Furthermore, by applying Theorems~\ref{Thm:Uniform-bound-normalized-constant} 
    and \ref{Thm:Unifrom-moment-estimate}, we have 
    \begin{align}
        \left|\III_\alpha^{(n)}(x)\right| 
        &\le \frac{n}{2}\left(\sup_{n \in \N}\sup_{\alpha \in K}\max_{x \in [0, 1]}\frac{1}{Z_{\alpha, x}^{(n)}}\right) 
        \|f''\|_\infty O\left(\frac{1}{n^2}\right) 
        = O\left(\frac{1}{n}\right)
        \label{Eq:difference of generators-3}
    \end{align}
    as $n \to \infty$. 
    On the other hand, for $\ve>0$, $x(1-x)<\ve$ for $x \in [0, \ve] \cup [1-\ve, 1]$. 
    Then, we have
    \begin{align}
        \left|\IV_\alpha^{(n)}(x)\right| 
        &\le \frac{\ve \|f''\|_\infty}{2\alpha}\max_{x \in [0, \ve] \cup [1-\ve, 1]}
        \left|\frac{1}{Z_{\alpha, x}^{(n)}}-1\right|
        +\frac{\|f''\|_\infty}{8\alpha}\max_{x \in [\ve, 1-\ve]}
        \left|\frac{1}{Z_{\alpha, x}^{(n)}}-1\right| \nn \\
        &\le \frac{\ve \|f''\|_\infty}{2\alpha}
        \sup_{n \in \N}\sup_{\alpha \in K}\max_{x \in [0, 1]}
        \left|\frac{1}{Z_{\alpha, x}^{(n)}}-1\right|
        +\frac{\|f''\|_\infty}{8\alpha}\sup_{\alpha \in K}\max_{x \in [\ve, 1-\ve]}
        \left|\frac{1}{Z_{\alpha, x}^{(n)}}-1\right| \nn \\
        &=\ve O(1)+o(1)
        \label{Eq:difference of generators-4}
    \end{align}
    as $n \to \infty$, by virtue of Proposition~\ref{Prop:Order-moment-fractional-binomial}~(1) and Theorem~\ref{Thm:Uniform-bound-normalized-constant}. 
    Therefore, by combining \eqref{Eq:difference of generators-0} 
    with \eqref{Eq:difference of generators-1}, 
    \eqref{Eq:difference of generators-2}, \eqref{Eq:difference of generators-3}, 
    and \eqref{Eq:difference of generators-4}, we obtain
    \[
        \sup_{\alpha \in K}\|n(B_{\alpha, n}-I)f-\mathcal{L}_\alpha f\|_\infty = \ve O(1)+o(1)
    \]
    as $n \to \infty$. Because $\ve>0$ is arbitrary, 
    we conclude \eqref{Eq:Voronovskaya-Bernstein} after letting $\ve \searrow 0$. 
\end{proof}

\begin{lm}\label{Lem:core}
    $\mathcal{D}_0$ is a core for $(\mathcal{L}_\alpha, \mathrm{Dom}(\mathcal{L}_\alpha))$. 
\end{lm}

\begin{proof}
    Because $C^2([0, 1])$ is a core for $(\mathcal{L}_\alpha, \mathrm{Dom}(\mathcal{L}_\alpha))$,
    it is sufficient to show that for $f \in C^2([0, 1])$,   
    there exists a family $(f_\ve)_{\ve>0} \subset \mathcal{D}_0$ such that 
    \begin{equation}\label{Eq:core-aim}
        \lim_{\ve \searrow 0}
        (\|\mathcal{L}_\alpha(f_\ve-f)\|_\infty + \|f_\ve-f\|_\infty)=0.
    \end{equation}
    In the sequel, we assume that 
    $f \in C^2([0, 1])$ satisfies $f'(0) \neq 0$ and $f'(1) \neq 0$. 
    If $f'(0)=0$ or $f'(1)=0$, then the proof is also valid by minor modification.
    Let $\ve>0$ be sufficiently small so that $\ve<1/2$ and 
    \begin{equation}\label{Eq:epsilon-condition}
        \ve\left(\log 2-\frac{1}{2}\right) < \min\{|f'(0)|, |f'(1)|\}.
    \end{equation}
    For $\delta \in (0, \ve]$, we define 
    \begin{align*}
        g_\delta(x)&:=\frac{\ve}{x+\delta}-\frac{\ve}{\ve+\delta},
        \qquad x \in [0, \ve], \\
        \widehat{g}_\delta(x)&:=\frac{\ve}{1-x+\delta}-\frac{\ve}{\ve+\delta}, 
        \qquad x \in [1-\ve, 1]. 
    \end{align*}
    The function $g_\delta$ (resp.~$\widehat{g}_\delta$) is non-negative and 
    continuous on $[0, \ve]$ (resp.~$[1-\ve, 1]$) with $g_\delta(\ve)=0$ 
    (resp.~$\widehat{g}_\delta(1-\ve)=0$). 
    We also define 
    \begin{align*}
        G(\delta)&:=\int_0^\ve g_\delta(x) \, \dd x
        =\ve \log \frac{\ve+\delta}{\delta}-\frac{\ve^2}{\ve+\delta},
        \qquad \delta \in (0, \ve], \\
        \widehat{G}(\delta)
        &:=\int_{1-\ve}^1 \widehat{g}_\delta(x) \, \dd x
        =\ve \log \frac{\ve+\delta}{\delta}-\frac{\ve^2}{\ve+\delta}, 
        \qquad \delta \in (0, \ve].
    \end{align*}
    Then, both $G$ and $\widehat{G}$ 
    are continuous on $(0, \ve]$
    and satisfy that 
    \[
    \begin{aligned}
        G(0+)&=+\infty, &\qquad G(\ve)&=\ve \log 2-\frac{\ve}{2}<|f'(0)|, \\
        \widehat{G}(0+)&=+\infty, &\qquad \widehat{G}(\ve)&=\ve \log 2-\frac{\ve}{2}<|f'(1)|
    \end{aligned}
    \]
    due to \eqref{Eq:epsilon-condition}. 
    Hence, the mean value theorem implies the existence of 
    $\delta_0, \delta_1 \in (0, \ve)$ such that $G(\delta_0)=|f'(0)|$ and $\widehat{G}(\delta_1)=|f'(1)|$. 
    We now define a family of continuous functions $(f_\ve)_{\ve>0} \subset C([0, 1])$ by 
    \[
        f_\ve(x):=\begin{cases}
        f(x)+\{\mathrm{sgn}\,f'(0)\}\dis 
        \int_x^\ve \int_y^\ve g_{\delta_0}(s) \, \dd s \, \dd y
        & \text{if }x \in [0, \ve), \\
        f(x) & \text{if }x \in [\ve, 1-\ve], \\
        f(x)+\{\mathrm{sgn}\,f'(1)\}\dis 
        \int_{1-\ve}^x \int_{1-\ve}^y \widehat{g}_{\delta_1}(s) \, \dd s \, \dd y
        & \text{if }x \in (1-\ve, 1]. \\
    \end{cases}
    \]
    Then, we have 
    \[
        f_\ve'(x):=\begin{cases}
        f'(x)-\{\mathrm{sgn}\,f'(0)\}\dis 
        \int_x^\ve g_{\delta_0}(s) \, \dd s  
        & \text{if }x \in [0, \ve), \\
        f'(x) & \text{if }x \in (\ve, 1-\ve), \\
        f'(x)-\{\mathrm{sgn}\,f'(1)\}\dis 
        \int_{1-\ve}^x \widehat{g}_{\delta_1}(s) \, \dd s  
        & \text{if }x \in (1-\ve, 1]. \\
    \end{cases}
    \]
    In particular, we see that 
    \begin{align*}
        f'_\ve(0)&=f'(0)-\mathrm{sgn}\{f'(0)\}|f'(0)|=0, \\
        f'_\ve(1)&=f'(1)-\mathrm{sgn}\{f'(1)\}|f'(1)|=0, \\
        f'_\ve(\ve\pm)&=f'(\ve), \qquad
        f'_\ve((1-\ve)\pm)=f'(1-\ve).
    \end{align*}
    We also have
    \[
        f_\ve''(x):=\begin{cases}
        f''(x)+\{\mathrm{sgn}\,f'(0)\}\dis 
        g_{\delta_0}(x) 
        & \text{if }x \in [0, \ve), \\
        f''(x) & \text{if }x \in (\ve, 1-\ve), \\
        f''(x)-\{\mathrm{sgn}\,f'(1)\}\dis 
        \widehat{g}_{\delta_1}(x) 
        & \text{if }x \in (1-\ve, 1] \\
    \end{cases}
    \]
    and 
    \[
        f_\ve''(\ve\pm)=f''(\ve), \qquad f_\ve''((1-\ve)\pm)=f''(1-\ve).
    \]
    These mean that $f_\ve \in \mathcal{D}_0$ for $\ve>0$. 
    Furthermore, we obtain that
    \begin{align*}
        |\mathcal{L}_\alpha(f_\ve-f)(x)|
        &=\left|\frac{1}{\alpha}x(1-x)g_{\delta_0}(x)\right| \\
        &=\frac{1}{\alpha}x(1-x)  \frac{\ve(\ve-x)}{(x+\delta_0)(\ve+\delta_0)} 
        \le \frac{(1-x)(\ve-x)}{\alpha} \le \frac{\ve}{\alpha}
    \end{align*}
    for $x \in [0, \ve)$ and 
    \begin{align*}
        |\mathcal{L}_\alpha(f_\ve-f)(x)|
        &=\left|\frac{1}{\alpha}x(1-x)\widehat{g}_{\delta_1}(x)\right| \\
        &=\frac{1}{\alpha}x(1-x)  \frac{\ve\{\ve-(1-x)\}}{(1-x+\delta_1)(\ve+\delta_1)} 
        \le \frac{x\{\ve-(1-x)\}}{\alpha} \le \frac{\ve}{\alpha}
    \end{align*}
    for $x \in (1-\ve, 1]$. 
    On the other hand, it holds that $|\mathcal{L}_\alpha(f_\ve-f)(x)| =0$ 
    for $x \in [\ve, 1-\ve]$. Hence, we have
    \begin{equation}\label{Eq:L_alpha-estimate}
        \|\mathcal{L}_\alpha(f_\ve-f)\|_\infty \le \frac{\ve}{\alpha},
        \qquad \ve>0. 
    \end{equation}
    We also observe that 
    \begin{align*}
        |f_\ve(x)-f(x)| &= \int_x^\ve \int_y^\ve g_{\delta_0}(s) \, \dd s \, \dd y \\ 
        &\le \int_0^\ve \int_0^\ve g_{\delta_0}(s) \, \dd s \, \dd y 
        =\int_0^\ve G(\delta_0) \, \dd y=\ve|f'(0)|
    \end{align*}
    for $x \in [0, \ve)$ and 
    \begin{align*}
        |f_\ve(x)-f(x)| &= \int_{1-\ve}^x \int_{1-\ve}^y \widehat{g}_{\delta_1}(s) \, \dd s \, \dd y \\ 
        &\le \int_{1-\ve}^1 \int_{1-\ve}^1 \widehat{g}_{\delta_1}(s) \, \dd s \, \dd y 
        =\int_{1-\ve}^1 \widehat{G}(\delta_1) \, \dd y=\ve|f'(1)|
    \end{align*}
    for $x \in (1-\ve, 1]$. Because $|f_\ve(x)-f(x)|=0$ for $x \in [\ve, 1-\ve]$ by definition, 
    we obtain 
    \begin{equation}\label{Eq:f_epsilon-estimate}
        \|f_\ve-f\|_\infty \le \ve \max\{|f'(0)|, |f'(1)|\}. 
    \end{equation}
    Therefore, \eqref{Eq:core-aim} follows from \eqref{Eq:L_alpha-estimate} 
    and \eqref{Eq:f_epsilon-estimate}.
\end{proof}

We can now prove Theorem~\ref{Thm:convergence-fractional-Bernstein-1}.

\begin{proof}[Proof of Theorem~{\rm \ref{Thm:convergence-fractional-Bernstein-1}}]
    It follows from Lemma \ref{Lem:core} and \cite[Proposition~1.3.1]{EK} that 
    $\mathrm{Range}(\lambda - \mathcal{L}_\alpha|_{\mathcal{D}_0})$ is dense in $C([0, 1])$
    for some $\lambda>0$. 
    Then, by combining this fact with Lemma \ref{Lem:Voronovskaya-Bernstein}, 
    we can apply Trotter's approximation theorem (cf.~\cite[Theorem~5.2]{Trotter}) 
    to conclude \eqref{Eq:semigroup-uniform-convergence}.
\end{proof}

Finally, we mention that the limiting semigroup 
$(T_\alpha(t))_{t \ge 0}$ obtained 
in Theorem~\ref{Thm:convergence-fractional-Bernstein-1} 
has the following probabilistic interpretation.

    \begin{pr}\label{Prop:diffusion semigroup}
        For $\alpha>0$ and $x \in [0, 1]$, let $(X_{\alpha, x}(t))_{t \ge 0}$ 
        be the one-dimensional diffusion process
        that is a unique strong solution to the stochastic differential equation 
        \begin{equation}\label{Eq:SDE}
            \dd X_{\alpha, x}(t)
            =\sqrt{\frac{1}{\alpha}X_{\alpha, x}(t)(1-X_{\alpha, x}(t))} \, \dd W(t), 
            \qquad X_{\alpha, x}(0)=x,
        \end{equation}
        where $(W(t))_{t \ge 0}$ is standard one-dimensional Brownian motion. 
        Then, we have 
        \[
            T_\alpha(t) f(x)=\E\big[f\big(X_{\alpha, x}(t)\big)\big], 
            \qquad f \in C([0,1]), \, x \in [0,1], \, t \ge 0.
        \]
    \end{pr}       

\begin{proof}
    Because \eqref{Eq:SDE} has a unique strong solution,
    $(X_{\alpha, x}(t))_{t \ge 0}$ is the unique 
    solution to the martingale problem for $(\mathcal{L}_\alpha, \delta_x)$ for every $x \in [0, 1]$. 
    Moreover, we note that the diffusion coefficient $(2\alpha)^{-1}x(1-x)$ of $\mathcal{L}_\alpha$ is continuous and bounded on $[0, 1]$. 
    Then, we can apply \cite[Theorem~32.11]{Kallenberg} to conclude that
    \[
        \widetilde{T}_{\alpha}(t)f(x):=\E[f(X_{\alpha, x}(t))], 
        \qquad f \in C([0, 1]), \, x \in [0, 1], \, t \ge 0
    \]
    is a Feller semigroup on $C([0, 1])$ and $\mathcal{L}_\alpha$
    uniquely extends to the associated generator. 
    Because a Feller semigroup is uniquely determined by 
    its generator (cf.~\cite[Lemma~17.5]{Kallenberg}), we have 
    $T_\alpha(t)=\widetilde{T}_\alpha(t)$ for all $t \ge 0$. 
\end{proof}

\subsection{The fractional Bernstein operator with $\alpha=\alpha(x)$ being a function of $x$}
\label{Subsect:limit theorems for fractional Bernstein-2}

In this section, we consider the fractional Bernstein operator whose parameter $\alpha > 0$
is replaced with a function $\alpha=\alpha(x) \colon [0, 1] \to (0, \infty)$. 
In what follows, we assume that there exist some positive constants 
$\alpha_L>0$ and $\alpha_U>0$ such that 
    \begin{equation}\label{Eq:alpha-bound}
        (0<)\,\alpha_L \le \alpha(x) \le \alpha_U, \qquad x \in [0, 1]. 
    \end{equation}
Let $n \in \N$ and $x \in [0, 1]$. 
Then, the $\alpha(x)$-fractional binomial distribution 
    \[
        \mu_{\alpha(x), x}^{(n)}(\dd z)=
        \frac{\alpha(x)}{Z_{\alpha(x), x}^{(n)}}\sum_{j=0}^n \binom{\alpha(x)n}{\alpha(x)j}x^{\alpha(x)j}
        (1-x)^{\alpha(x)(n-j)}\delta_j(\dd z)
    \]
with \eqref{Eq:alpha-bound} turns out to have almost the same properties 
as those of $\mu_{\alpha, x}^{(n)}$ 
discussed in Section~\ref{Sect:fractional-binomial-distribution}.
Indeed, the fractional-order Taylor series 
(Proposition~\ref{Prop:Fractional-Taylor-series-HH}) is still valid 
if we replace $\alpha$ with $\alpha(x)$ satisfying \eqref{Eq:alpha-bound}. 
Hence, Theorems~\ref{Thm:Moment-fractional-binomial-distribution} and \ref{Thm:Unifrom-moment-estimate}, 
Proposition~\ref{Prop:Order-moment-fractional-binomial}, and
Corollaries~\ref{Cor:Mean-variance-fractional-binomial} and \ref{Cor:Moment-estimate} also hold for 
the $\alpha(x)$-fractional binomial distribution. 
In particular, Theorem~\ref{Thm:Uniform-bound-normalized-constant} and \eqref{Eq:alpha-bound} 
imply the existence of two constants 
$C_{4,1}>0$ and $C_{4,2}>0$ independent of $n \in \N$ such that 
    \[
        C_{4,1} \le Z_{\alpha(x), x}^{(n)} \le C_{4,2}, \qquad x \in [0, 1]. 
    \]

We now assume that $\alpha(x)$ is continuous in $x$. 
Let us consider the $\alpha(x)$-fractional Bernstein operator $B_{\alpha(\cdot), n}$ given by
        \begin{align*}
        B_{\alpha(\cdot), n}f(x)
        &=\int_{\R} f\left(\frac{z}{n}\right) \, \mu_{\alpha(x), x}^{(n)}(\dd z) \\
        &=\sum_{j=0}^n \frac{\alpha(x)}{Z_{\alpha(x), x}^{(n)}}
        \binom{\alpha(x)n}{\alpha(x)j}x^{\alpha(x)j}(1-x)^{\alpha(x)(n-j)}f\left(\frac{j}{n}\right)
        \end{align*}
for $f \in C([0, 1])$ and $x \in [0, 1]$. 
We also consider the second-order differential operator
    \[
        \mathcal{L}_{\alpha(\cdot)}f(x)=\frac{1}{2\alpha(x)}x(1-x)f''(x)
    \]
for $f \in C^2([0, 1])$ and $x \in [0, 1]$. 
Then, the operator $(\mathcal{L}_{\alpha(\cdot)}, C^2([0, 1]))$ is also 
closable on $C([0, 1])$, and its closure 
$(\mathcal{L}_{\alpha(\cdot)}, \mathrm{Dom}(\mathcal{L}_{\alpha(\cdot)}))$
generates the Feller semigroup $(T_{\alpha(\cdot)}(t))_{t \ge 0}$ 
on $C([0, 1])$, which is no longer the Wright--Fisher diffusion semigroup 
in general (see \cite[Theorems~1.6.11 and 6.2.6]{Alt-C}). 

By following the proof of Theorem \ref{Thm:convergence-fractional-Bernstein-1},
we also prove that the appropriately scaled operators $(B_{\alpha(\cdot), n}^{\lfloor nt \rfloor})_{n \in \N_0}$, $t \ge 0$,  defined by the iterates 
of $B_{\alpha(\cdot),n}$ converges to $(T_{\alpha(\cdot)}(t))_{t \ge 0}$.  

\begin{tm}
\label{Thm:convergence-fractional-Bernstein-2}
    Suppose that a function $\alpha=\alpha(x) \colon [0, 1] \to (0, \infty)$ is continuous. 
    Then, we have 
        \begin{equation}\label{Eq:semigroup-uniform-convergence-2}
            \lim_{n \to \infty}
            \|B_{\alpha(\cdot),n}^{\lfloor nt \rfloor}f- 
            T_{\alpha(\cdot)}(t)f\|_\infty=0, 
            \qquad f \in C([0, 1]), \, t \ge 0.
        \end{equation}
\end{tm} 

\begin{proof}
    From Lemma~\ref{Lem:Voronovskaya-Bernstein}, 
    it holds that 
        \[
            \lim_{n \to \infty}\|n(B_{\alpha(\cdot), n}-I)f 
            - \mathcal{L}_{\alpha(\cdot)}f\|_\infty=0
        \]
    for $f \in C^2([0, 1])$, by noting \eqref{Eq:alpha-bound}.
    On the other hand, Lemma~\ref{Lem:core} tells us that 
    for $f \in C^2([0, 1])$, there exists 
    a family $(f_\ve)_{\ve>0} \subset \mathcal{D}_0$ satisfying \eqref{Eq:core-aim} with $\alpha$ replaced by $\alpha_L$.
    Then, \eqref{Eq:alpha-bound} leads to
        \[
            \|\mathcal{L}_{\alpha(\cdot)}(f_\ve-f)\|_\infty+\|f_\ve-f\|_\infty
            \le \|\mathcal{L}_{\alpha_L}(f_\ve-f)\|_\infty+\|f_\ve-f\|_\infty \to 0
        \]
    as $\ve \searrow 0$, which implies that $\mathcal{D}_0$ is also a core for 
    $(\mathcal{L}_{\alpha(\cdot)}, \mathrm{Dom}(\mathcal{L}_{\alpha(\cdot)}))$.
    Therefore, it follows from \cite[Proposition~1.3.1]{EK} 
    that $\mathrm{Range}(\lambda-\mathcal{L}_{\alpha(\cdot)}|_{\mathcal{D}_0})$
    is dense in $C([0, 1])$ for some $\lambda>0$. 
    Then, Trotter's approximation theorem allows us 
    to obtain \eqref{Eq:semigroup-uniform-convergence-2}. 
\end{proof}

Furthermore, we obtain a probabilistic interpretation 
of Theorem~\ref{Thm:convergence-fractional-Bernstein-2}
in terms of a diffusion process corresponding to $\mathcal{L}_{\alpha(\cdot)}$.
Consequently, we can capture more-general diffusion processes on $[0, 1]$ 
via Theorem~\ref{Thm:convergence-fractional-Bernstein-2}
and Proposition~\ref{Prop:diffusion semigroup-2} as follows.

\begin{pr}\label{Prop:diffusion semigroup-2}
    We assume that 
    $\alpha=\alpha(x) \colon [0, 1] \to (0, \infty)$
    is continuous and $\alpha(x)^{-1}=1/\alpha(x)$ is 
    Lipschitz continuous on $[0, 1]$. 
    For $x \in [0, 1]$, let 
    $(X_{\alpha(\cdot), x}(t))_{t \ge 0}$ be the 
    one-dimensional diffusion process
    that solves the stochastic differential equation 
        \begin{equation}\label{Eq:SDE-2}
            \dd X_{\alpha(\cdot), x}(t)
            =\sqrt{\frac{1}{\alpha(X_{\alpha(\cdot), x}(t))}X_{\alpha(\cdot), x}(t)(1-X_{\alpha(\cdot), x}(t))} \, \dd W(t), 
            \qquad X_{\alpha(\cdot), x}(0)=x,
        \end{equation}
    where $(W(t))_{t \ge 0}$ is standard one-dimensional Brownian motion. 
    Then, we have 
        \begin{equation*}\label{Eq:Bernstein-coinidence-2}
            T_{\alpha(\cdot)}(t) f(x)=\E\big[f\big(T_{\alpha(\cdot), x}(t)\big)\big], 
            \qquad f \in C([0,1]), \, x \in [0,1], \, t \ge 0.
        \end{equation*}
\end{pr}

\begin{proof}
We omit the details of the proof because it is essentially same as that of Proposition~\ref{Prop:diffusion semigroup}.
On the other hand, we should note that \eqref{Eq:SDE-2} has the unique strong solution. 
Indeed, the existence of the weak solution to \eqref{Eq:SDE-2}
holds true since the diffusion coefficient 
    \[
        \sigma(x):=\sqrt{\alpha(x)^{-1}x(1-x)}, \qquad x \in [0, 1],
    \]
of \eqref{Eq:SDE-2} is bounded and continuous on $[0, 1]$. 
Moreover, it follows from the Lipschitz continuity of $1/\alpha(x)$ that 
    \begin{align*}
        |\sigma(x)-\sigma(y)|
        &=\left|\sqrt{\alpha(x)^{-1}x(1-x)} - \sqrt{\alpha(y)^{-1}y(1-y)} \right|\\
        &\le \sqrt{\left|\alpha(x)^{-1}x(1-x) - \alpha(y)^{-1}y(1-y) \right|} \\
        &\le \sqrt{\left|\alpha(x)^{-1}-\alpha(y)^{-1} \right|x(1-x)+\alpha(y)^{-1}|x(1-x)-y(1-y)|} \\
        &\le \sqrt{C_{4,3}|x-y|+3\alpha_L^{-1}|x-y|} 
        \le C_{4,4}|x-y|^{1/2}, \qquad x, y \in [0, 1],
    \end{align*}
which implies the pathwise uniqueness of \eqref{Eq:SDE-2}. 
Therefore, the Yamada--Watanabe theorem (cf.~\cite[Theorem~1]{YW})
leads to the existence of the unique strong solution to \eqref{Eq:SDE-2}. 
\end{proof}

\smallskip
\noindent
\textbf{Acknowledgements.}  
The first-named author was supported by 
JSPS KAKENHI Grant Numbers 19H00643 and 25K07056 and The Kyoto University Foundation. 
The second-named author was supported by
JSPS KAKENHI Grant Numbers 23K12986 and 23K03155.

\bigskip
\noindent
\textbf{Data availability} 
No datasets were generated or analyzed during the current study.

\section*{Declarations}

\noindent
\textbf{Conflict of interest} 
On behalf of all authors, the corresponding author states that there is no conflict of interest.

\end{document}